\input ./style/arxiv-general.cfg
\documentclass[aop,MSNbibl,seceqn,dvips]{arximspdf}
\makeatletter
   \@ifpackageloaded{graphicx}{}{\usepackage{graphicx}}
\makeatother

\doi{10.1214/15-AOP1016}
\volume{44}
\issue{3}
\pubyear{2016}
\firstpage{1985}
\lastpage{2023}
\docsubty{FLA}

\makeatletter
\newcommand{\rrvert}{\vert}
\newcommand{\rrVert}{\Vert}
\newcommand{\llvert}{\vert}
\newcommand{\llVert}{\Vert}
\newtheorem{theorem}{Theorem}[section]
\newtheorem{corollary}[theorem]{Corollary}
\newproclaim{definition}[theorem]{Definition}
\newtheorem{lemma}[theorem]{Lemma}
\newtheorem{proposition}[theorem]{Proposition}
\newproclaim{remark}[theorem]{Remark}
\def\R{\mathbb R}
\def\N{\mathbb N}

\def\E{\mathbb E}

\newcommand{\eqref}[1]{(\ref{#1})}
\makeatother

\begin{document}
\begin{frontmatter}

\title{Strong uniqueness for SDEs in Hilbert spaces with nonregular drift}
\runtitle{SDE's with nonregular drift}

\begin{aug}
\author[A]{\fnms{G.}~\snm{Da Prato}\corref{}\ead[label=e1]{giuseppe.daprato@sns.it}\thanksref{T1,m1}},
\author[B]{\fnms{F.}~\snm{Flandoli}\ead[label=e2]{flandoli@dma.unipi.pi}\thanksref{T2,m2}},
\author[C]{\fnms{M.}~\snm{R\"{o}ckner}\ead[label=e3]{roeckner@math.uni-bielefeld.de}\thanksref{T3,m3}} 
\and
\author[D]{\fnms{A. Yu.}~\snm{Veretennikov}\ead[label=e4]{a.veretennikov@leeds.ac.uk}\thanksref{m41,m42,m43,T4}}
\runauthor{Da Prato, Flandoli, R\"{o}ckner and Veretennikov}
\thankstext{T1}{Supported in part by Progetto GNAMPA 2013.}
\thankstext{T2}{Supported in part by the Italian Ministero Pubblica Istruzione, PRIN 2010--2011,
``Problemi differenziali di evoluzione: approcci
deterministici e stocastici e loro interazioni.''}
\thankstext{T3}{Supported in part by the DFG through SFB-701 and IRTG 1132.}
\thankstext{T4}{Supported in part by
the Russian Foundation for Basic Research Grant 13-01-12447 $ofi_m$.}
\affiliation{Scuola Normale Superiore,\thanksmark{m1} Universit\`a di Pisa,\thanksmark{m2}
University of Bielefeld\thanksmark{m3}, University of Leeds\thanksmark{m41}, National Research University Higher
School of Economics Moscow\thanksmark{m42} and Institute for Information Transmission Problems Moscow\thanksmark{m43}}
%
\address[A]{G. Da Prato\\
Scuola Normale Superiore\\
Piazza dei Cavalieri 7\\
56136 Pisa\\
Italy\\
\printead{e1}}
\address[B]{F. Flandoli\\
Universit\`a di Pisa\\
Largo Pontecorvo 5\\
56127 Pisa\\
Italy\\
\printead{e2}\hspace*{19pt}}
\address[C]{M. R\"ockner\\
University of Bielefeld\\
Universit\`atsstrasse 25\\
33615 Bielefeld\\
Germany\\
\printead{e3}}
\address[D]{A. Yu. Veretennikov\\
University of Leeds\\
LS2 9JT, Woodhouse Lane\\
Leeds LS2 9JT\\
UNited Kingdom\\
and\\
Institute for Information\\
\quad Transmission Problems\\
Moscow\\
Russia\\
and\\
National Research University\\
\quad Higher School of Economics\\
Moscow\\
Russia\\
\printead{e4}}
\end{aug}

%
\received{\smonth{4} \syear{2014}}
%
\revised{\smonth{1} \syear{2015}}

%
\begin{abstract}
We prove pathwise uniqueness for a class of stochastic differential
equations (SDE) on a Hilbert space with cylindrical Wiener noise, whose
nonlinear drift parts are sums of the sub-differential of a convex
function and a bounded part. This generalizes a classical result by one
of the authors to infinite dimensions. Our results also generalize and
improve recent results by N.~Champagnat and P. E. Jabin, proved in
finite dimensions, in the case where their diffusion matrix is constant
and nondegenerate and their weakly differentiable drift is the (weak)
gradient of a convex function.
We also prove weak existence, hence obtain unique strong solutions by
the Yamada--Watanabe theorem. The proofs are based in part on a recent
maximal regularity result in infinite dimensions, the theory of
quasi-regular Dirichlet forms and an infinite dimensional version of a
Zvonkin-type transformation. As a main application, we show pathwise
uniqueness for stochastic reaction diffusion equations perturbed by a
Borel measurable bounded drift. Hence, such SDE have a unique strong solution.
\end{abstract}

%
\end{frontmatter}

\section{Introduction}\label{sec1}

In a separable Hilbert space $H$, with inner product $ \langle
\cdot,\cdot \rangle$ and norm $\llvert \cdot\rrvert $, we
consider the SDE%
%
\begin{eqnarray}\label{SDE}
dX_{t} & =& \bigl( AX_{t}-\nabla V ( X_{t} ) +B
( X_{t} ) \bigr) \,dt+dW_{t},
\nonumber
\\[-8pt]
\\[-8pt]
\nonumber
X_{0} & =&z,
\end{eqnarray}
where we assume:

(H1) $A\dvtx D ( A ) \subset H\rightarrow H$ is a
self-adjoint and strictly
negative definite operator (i.e., $A\le-\omega I$ for some $\omega
>0$), with $A^{-1}$ of trace class.

(H2) $V\dvtx H\rightarrow (-\infty ,+\infty ] $ is a
\textit{convex, proper,
lower-semicontinuous, lower bounded} function; denote by $D_{V}$ the
set of all $x\in\{V<\infty\}$ such
that $V$ is G\^{a}teaux differentiable at $x$.

(H3) For the Gateaux derivative $\nabla V$ we have for some
$\varepsilon>0$
\begin{eqnarray}\label{e2}
\gamma ( D_{V} ) &=&1,
\nonumber\\
\int_{H} \bigl( \bigl|V(x)\bigr|^{2+\varepsilon}+\bigl\vert\nabla V(x)
\bigr\vert^{2} \bigr)\gamma ( dx )& <&\infty,\\
 \int_{H}
\bigl\llVert D^{2}V ( x ) \bigr\rrVert _{\mathcal{L} ( H ) }\nu ( dx ) &<&
\infty,\nonumber
\end{eqnarray}
where $\gamma$ is the centered Gaussian measure in $H$ with covariance
$Q=-\frac{1}{2}A^{-1}$ and $\nu$ is the probability measure on $H$
defined as
\[
\nu ( dx ) =\frac{1}{Z}e^{-V ( x ) }\gamma ( dx ) ,\qquad Z=\int
_{H}e^{-V ( x ) }\gamma ( dx ).
\]
Clearly, $\gamma$ and $\nu$ have the same zero sets.
Here, the second assumption in \eqref{e2} means that there exists
$u_n\in\mathcal F C^2_b(H), n\in\mathbb N$, such that
$V=\lim_{n\to\infty} u_n$ in $L^2(H,\nu)$ and $D^2V:=\lim_{n\to\infty}
D^2u_n$ in $L^2(H,\nu;L(H))$, where $\mathcal F C^2_b(H)$ denotes the
set of all $C^2_b$-cylindric functions on $H$ (see below for the
precise definition) and $L(H)$ the set of all bounded linear operators
from $H$ to $H$.

(H4) $B\dvtx H\rightarrow H$ is Borel measurable and
bounded.

(H5) $W$ is an ($\mathcal F_t$)-cylindrical Brownian motion
in $H$, on some pobability space $(\Omega,\mathcal F,P)$ with normal
filtration $(\mathcal F_t), t\ge0$.

 Formally, $W$ is a process of the form $W_{t}=\sum_{i=1}^{\infty
}W_{t}^{i}e_{i}$ where $W_{t}^{i}$ are independent real valued Brownian
motions defined on a probability space $ ( \Omega,\mathcal
{F},P ) $
and $ \{ e_{i} \} _{i\in\mathbb{N}}$ is a complete orthonormal
system in $H$; for every $h\in H$, the series $ \langle W_{t}
,h \rangle=\sum_{i=1}^{\infty}W_{t}^{i} \langle e_{i},h
\rangle
$ converges in $L^{2} ( \Omega ) $.
%
\begin{remark}
\label{r0}
Since $A$ is strictly negative definite, we may assume $V(x)\ge
\varepsilon|x|^2, x\in H$, for some $\varepsilon>0$ and all $x\in H$.
Otherwise, replace $A$ by $A+\frac{\omega}2 I$ and $V$ by $V+\frac
{\omega}2 |x|^2+|\inf_{x\in H}V(x)|$. In particular, without loss of
generality we have that $|x|^pe^{-V(x)}$ is bounded in $x\in H$ for all
$p\in(0,\infty)$.
\end{remark}

\begin{remark}
\label{r1}
(i) We note that if $x\in D_V$ by definition
\[
\lim_{s\to0} \frac{1}s \bigl(V(x+sh)-V(x)\bigr)=\bigl
\langle\nabla V(x),h \bigr\rangle
\]
for all $h\in H$ where a priori the limit is taken in the Alexandrov
topology on $(-\infty,+\infty]$, since $V(x+sh)$ could be $+\infty$ for
some $s$. On the other hand, the limit $\langle\nabla V(x),h \rangle\in
\mathbb R$, so $V(x+sh)\in \mathbb R$ for $s\le s_0$ for some small
enough $s_0>0$.

(ii) If $\{V<\infty\}$ is open, then $\gamma(D_V)=1$. Indeed,
if $\{V<\infty\}$ is open, then $V$ is continuous on $\{V<\infty\}$;
see, for example, \cite{P93}, Proposition~3.3. Since furthermore, $V$
is then locally Lipschitz on $\{V<\infty\}$ (see, e.g., \cite{P93},
Proposition~1.6), it follows by the fundamental result in \cite
{A76,P78}; see also \cite{Bo10}, Section~10.6, that $\gamma(\{
V<\infty\}\setminus D_V)=0$. But $\gamma(\{V<\infty\})=1$, since $V\in
L^2(H,\gamma)$.
\end{remark}

It turns out that the condition on the second (weak) derivative in
\eqref{e2} in Hypothesis (H3) is too strong for some applications (see
Section~\ref{sec7} below). Therefore, we shall also consider the following
modified version of (H3):

(H3)$'$ $V$ and $\nabla V$ satisfy (H3) with the condition on
the second derivative of $V$ replaced by the following: there exists a
separable Banach space $E\subset H$, continuously and densely embedded,
such that $E\subset D(V)$, $\gamma(E)=1$ and on $E$ the function $V$ is
twice G\^ateaux-differentiable such that for all $x\in E$ its second G\^
ateaux-derivative $V''_E(x)\in L(E,E')$ (with $E'$ being the dual of
$E$) extends by continuity to an element in $L(H,E')$ such that
\[
\bigl\|V''_E(x)\bigr\|_{ L(H,E')}\le
\Psi\bigl(|x|_E\bigr)
\]
for some convex function $\Psi\dvtx [0,\infty)\to[0,\infty)$. Furthermore,
for $\gamma$-a.e. initial condition $z\in E$ there exists a
(probabilistically) weak solution $X^V=X^V(t), t\in[0,T]$, to SDE
\eqref{SDE} with $B=0$ so that
%
\begin{equation}
\label{e2'} \E\int_0^T\Psi
\bigl(\bigl|X^V(s)\bigr|_E\bigr) \,ds<\infty.
\end{equation}

Though (H3)$'$ is quite complicated to formulate, it is exactly what is
fulfilled if $\nabla V$ is a polynomial. We refer to Section~\ref{sec7.1} below.
%
\begin{remark}
\label{r2'}
We would like to stress at this point that the conditions on the
second derivative of $V$ both in (H3) and in (H3)$'$ are only used to be
able to apply the mean value theorem in the proof of Lemma~\ref{l38}
below. For the rest of this paper, we assume that (H1), (H2), (H4),
(H5) and (H3) or (H3)$'$ are in force.
\end{remark}

\begin{definition}
\label{Def sol}A solution of the SDE (\ref{SDE}) in $H$ is a filtered
probability space $ ( \Omega,\mathcal{F}, ( \mathcal
{F}_{t} )
_{t\geq0},P ) $ on $H$, an $H$-cylindrical $(\mathcal
F_t)$-Brownian motion $ ( W_{t} )
_{t\geq0}$ w.r.t. this space, a continuous $(\mathcal{F}_t)$-adapted
process $ (
X_{t}  ) _{t\geq0}$ on this space such that:
\begin{longlist}[(ii)]
\item[(i)] $X_{s}\in D_{V}$ for $dt\otimes P$ a.e. $ ( s,\omega ) $
and $\int_{0}
^{T}\llvert  \langle\nabla V ( X_{s} ) ,h \rangle
\rrvert  \,ds<\infty$ with probability one, for every $T>0$ and $h\in
D ( A ) $;

\item[(ii)] for every $h\in D ( A ) $ and $t\geq0$, one has
\[
\langle X_{t},h \rangle= \langle z,h \rangle+\int
_{0} ^{t} \bigl( \langle X_{s},Ah
\rangle+ \bigl\langle B ( X_{s} ) -\nabla V ( X_{s} ) ,h
\bigr\rangle \bigr) \,ds+ \langle W_{t},h \rangle
\]
with probability one.
\end{longlist}

If $X$ is $\mathcal{F}^{W}$-adapted, where $\mathcal{F}^{W}= (
\mathcal{F}_{t}^{W} ) _{t\geq0}$ is the normal filtration
generated by $W$, we say that $X$ is a strong solution.
\end{definition}

The Gaussian measure $\gamma$ is invariant for the linear equation
\[
dZ_{t}=AZ_{t}\,dt+dW_{t}
\]
while $\nu$ is invariant for the nonlinear equation
\[
dX_{t}= \bigl( AX_{t}-\nabla V ( X_{t} ) \bigr)
\,dt+dW_{t}.
\]
They are equivalent, since $V<\infty$ (hence $e^{-V}>0$) at least on $D_{V}$
and $\gamma ( D_{V} ) =1$. Hence, the full measure sets in $H$ are
the same with respect to $\gamma$ or $\nu$. Our main uniqueness result
is the following.

\begin{theorem}
\label{t4}There is a Borel set $\Xi\subset H$ with $\gamma (
\Xi ) =1$ having the following property. If $z\in\Xi$ and $X$, $Y$ are
two solutions with initial condition $x$ (in the sense of Definition~\ref{Def sol}), defined on the same filtered probability space $ (
\Omega,\mathcal{F}, ( \mathcal{F}_{t} ) _{t\geq0},P ) $ and
w.r.t. the same cylindrical Brownian motion $W$, then $X$ and $Y$ are
indistinguishable processes. Hence, by the Yamada--Watanabe theorem
they are (probabilistically) strong solutions and have the same law.
\end{theorem}

The proof is given in Section~\ref{sect proof uniqueness}. This result was
first proved in \cite{DFPR13} in the case $V=0$ (see also the more
recent \cite{DFPR14}, where also the case $V=0$, but with $B$ only
bounded on balls was treated) with a rather complex
proof based on the very nontrivial maximal regularity results in
$L^{p} ( H,\gamma ) $ for the Kolmogorov equation
\[
( \lambda-\mathcal{L}_{A,B} ) u=f
\]
associated to the SDE, where $\mathcal{L}_{A,B}$ is the operator formally
defined as
\[
\mathcal{L}_{A,B}u ( x ) =\tfrac{1}{2} \operatorname{ Tr} \bigl(
D^{2}u ( x ) \bigr) + \bigl\langle Ax+B ( x ) ,Du \bigr\rangle
\]
on suitable functions $u$, for $x\in D ( A ) $. Here, we
present a
much simpler proof which covers also the case $V\neq0$, based on
several new
ingredients.

First, in order to perform a suitable change of coordinates (analogous
to \cite{DFPR13} and \cite{DFPR14}), we use the family of Kolmogorov
equations
\[
( \lambda+\lambda_{i}-\mathcal{L}_{A,B,V} ) u=f
\]
or in vector form
%
\begin{equation}
( \lambda-A-\mathcal{L}_{A,B,V} ) U=F, \label{PDE vector}%
\end{equation}
where $\mathcal{L}_{A,B,V}$ is the operator formally defined as
\[
\mathcal{L}_{A,B,V}u ( x ) =\tfrac{1}{2} \operatorname{ Tr} \bigl(
D^{2}u ( x ) \bigr) + \bigl\langle Ax-\nabla V ( x ) +B ( x ) ,Du
\bigr\rangle
\]
on suitable functions $u$. The presence of the
term $\lambda_{i}u$ in the equation adds the advantages of the
resolvent of
$A$ [given by $ ( \lambda-A ) ^{-1}$] to those of the elliptic
regularity theory (given by $\mathcal{L}_{A,B}$). Moreover, we use the recent
maximal regularity results in $L^{2} ( H,\nu ) $ for the Kolmogorov
equation
\[
( \lambda-\mathcal{L}_{A,B,V} ) u=f
\]
proved in \cite{DL14}.

Second, thanks to the previous new Kolmogorov equation, we may apply a trick
based on It\^{o}'s formula and the multiplication by the factor $e^{-A_{t}}$
(see below the definition of $A_{t}$) which greatly simplifies the proof.

Third, we use Girsanov's theorem in a better form in the proof of the
main Lemma~\ref{l38}. The new proof of the lemma along with the previous two
innovations allow us to use only the $L^{2}$ theory of the Kolmogorov
equation, which is much simpler.

Fourth, we heavily use the theory of classical (gradient type)
Dirichlet forms on infinite dimensional state spaces.

For more background literature in the finite dimensional case following
the initiating work \cite{V80}, we refer
to \cite{DFPR13,DFPR14}. We only mention here the recent work~\cite{CJ13}, where SDEs with weakly differentiable drifts are studied.
In the case when in
\cite{CJ13} the diffusion matrix is constant and nondegenerate and if
the weakly differentiable drift is the (weak) gradient of a convex
function, our results generalize those in \cite{CJ13} from $\R^d$ to a
separable Hilbert space as state space, and to the case when a bounded
merely measurable drift part is added.
Finally, we mention the paper \cite{CD13} which concerns pathwise
uniqueness for some H\"older perturbation of reaction-diffusions
equations studied in spaces of continuous functions instead of square
integrable function.

The organization of the paper is as follows: Section~\ref{sec2} is devoted to
existence of solutions and Section~\ref{sec3} to the regularity theory of the
Kolmogorov operator~\eqref{PDE vector} above. The mentioned change of
coordinates is performed in Section~\ref{sec4}. Sections~\ref{sect proof uniqueness} and \ref{sec6} contain the
proof of our main Theorem~\ref{t4}. In Section~\ref{sec7}, we present applications.

We end this section by giving the definition of Sobolev spaces and some
notation. We consider an orthonormal basis $\{ e_k\dvtx  k\in\mathbb N\}$
of $H$ which diagonalizes $Q$ and set $Qe_k=\lambda_ke_k$ and $x_k =
\langle x, e_k\rangle$ for each $x\in H$, $k\in\mathbb N$. We denote
by $P_n$ the orthogonal projection on the linear span of $e_1, \ldots, e_n$.
For each $k\in\mathbb N \cup\{+\infty\}$, we denote by ${\mathcal
F}{\mathcal C}^k_b(H)$ the set of the cylindrical functions $\varphi(x)
= \phi(x_1, \ldots, x_n)$ for some $n\in\mathbb N$, with $\phi\in
C^k_b(\mathbb R^n)$.

For $\mu=\gamma$ or $\mu=\nu$, the Sobolev spaces $W^{1,2}(H,\mu)$ is
the completion of ${\mathcal F}{\mathcal C}^1_b(H)$ in the norm
\[
\|\varphi\|_{W^{1,2} (H, \mu)}^2 : = \int_H
\bigl(|\varphi|^2 + \|D\varphi \|^2\bigr)\,d\mu= \int
_H \Biggl(|\varphi|^2 + \sum
_{k=1}^\infty( D_k\varphi)^2
\Biggr) \,d\mu.
\]
The Sobolev spaces $W^{2,2}(H,\mu)$ is the completion of ${\mathcal
F}{\mathcal C}^2_b(H)$ in the norm
\begin{eqnarray*}
\| u \|_{W^{2,2} (H, \mu)}^2 &=& \| u
\|_{W^{1,2} (H, \mu
)}^2 + \int_H \operatorname{Tr} \bigl(
\bigl[D^2u\bigr]^2\bigr)\,d\mu
\\
&=& \| u \|_{W^{1,2} (H, \mu)}^2+ \sum_{h,k\in\mathbb N}
(D_{hk}u)^2 \,d\mu. %
\end{eqnarray*}
We denote the Borel $\sigma$-algebra on $H$ by $\mathcal B(H)$ and by
$B_b(H)$ the set of all bounded $\mathcal B(H)$-measurable functions
$\varphi\dvtx H\to\R$. We set for a function $\varphi\dvtx H\to\R$
\[
\|\varphi\|_\infty:=\sup_{x\in H}\bigl |\varphi(x)\bigr|.
\]
$I\dvtx H\to H$ denotes the identity operator on $H$. For $k\in\N$,
$C^k_b(H)$ denotes the set of all $\varphi\dvtx H\to\R$ of class $C^k$,
which together with all their derivatives up to order $k$ are bounded
and uniformly continuous. Furthermore, we reserve the symbol $D$ for
the closure of the derivative for $u\in\mathcal FC^1_b$ in $L^2(H,\mu
;H)$ for $\mu=\gamma$ or $\mu=\nu$. For the G\^ateaux derivative, we
use the symbol $\nabla$. Since they coincide on convex and Lipschitz
functions $u$, in the sense that $\nabla u$ is a $\gamma$- or $\nu
$-version of $Du$, we shall write $\nabla u$, whenever we want to
stress that we consider that special version.

\section{Existence}\label{sec2}
In this section, we shall prove that under conditions (H1)--(H4) from
the \hyperref[sec1]{Introduction}, which will be in force in all of this paper, that the SDE
\eqref{SDE} has a solution in the sense of Definition~\ref{Def sol}. We
start with the following proposition showing that the gradient $DV$ in
$L^2(H,\gamma;H)$ and the G\^ateaux derivative $\nabla V$ coincide
$\gamma$-a.e.
%
\begin{proposition}
\label{p2.1}
We have $V\in W^{1,2}(H,\gamma)$ and
\[
DV=\nabla V,\qquad \gamma\mbox{-a.e.}
\]
\end{proposition}
The proof of Proposition~\ref{p2.1} requires a numbers of lemmas.
%
\begin{lemma}
\label{l2.2}
Let $k\in Q^{1/2}H$. Then
\[
\lim_{s\to0} \frac{V(\cdot+sk)-V(\cdot)}{s}=\langle\nabla V,k\rangle \qquad
\mbox{in } L^2(H,\gamma).
\]
\end{lemma}
\begin{pf} Let $x\in\{V<\infty\}$. Then by convexity for $s\in(0,1)$
\[
V(x+sk)\le sV(x+k)+(1-s)V(x),
\]
hence
%
\begin{equation}
\label{e2.1} \frac{V(x+sk)-V(x)}{s}\le V(x+k)-V(x).
\end{equation}
Since $k\in Q^{1/2}H$, by the Cameron--Martin theorem (see, e.g., \cite{D04},
Section~1.2.3) the function on the right as a function of $x$ is
in $L^2(H,\gamma)$, since by assumption (H3) $V\in L^{2+\varepsilon
}(H,\gamma)$.

Furthermore for $x\in D_V$ taking the limit $s\to0$ in \eqref{e2.1} we
find that
\[
\bigl\langle\nabla V(x),k \bigr\rangle\le V(x+k)-V(x).
\]
Replacing $k$ by $sk$ which is also in $Q^{1/2}H$, and dividing by $s$,
we obtain
%
\begin{equation}
\label{e2.2} \bigl\langle\nabla V(x),k \bigr\rangle\le\frac{V(x+sk)-V(x)}{s}.
\end{equation}
But the left-hand side as a function of $x$ is in $L^{2+\varepsilon
}(H,\gamma)$ by assumption (H3).
Hence, \eqref{e2.1} and \eqref{e2.2} imply the assertion of the lemma
by Lebesgue's dominated convergence theorem, since $\gamma(D_V)=1$.
\end{pf}

Before we proceed to Lemma~\ref{l2.3} we need to introduce the
following space:
%
\begin{eqnarray}
\label{e5'} %
&& \mathcal D_0:= \biggl\{
u\in L^2(H,\gamma)\dvtx \exists F_u\in
L^2(H,\gamma ;H) \mbox{ such that}
\nonumber
\\[-8pt]
\\[-8pt]
\nonumber
&&\hspace*{36pt}\lim_{s\to0} \frac{1}{s} \bigl[u(\cdot
+se_i)-u(\cdot)\bigr]= \langle\nabla F_u,e_i
\rangle \mbox{ in } L^2(H,\gamma), \forall i\in\mathbb N \biggr\}.
\end{eqnarray}
Set $\widetilde D u:=F_u$ for $u\in\mathcal D_0$. Then obviously
$\mathcal F C^2_b\subset\mathcal D_0$ and $D\varphi=\widetilde D\varphi
$ for all $\varphi\in\mathcal F C^2_b$.
%
\begin{lemma}
\label{l2.3}
\textup{(i)} Let $u\in\mathcal D_0$, $\varphi\in\mathcal F C^2_b$ and $i\in
\mathbb N$. Then
\begin{eqnarray*}
\int_H \bigl\langle
\widetilde D u(x),e_i \bigr\rangle\varphi(x) \gamma(dx)&=& -\int
_H u(x)\bigl\langle D\varphi(x),e_i \bigr
\rangle \gamma(dx)
\nonumber
\\[-8pt]
\\[-8pt]
\nonumber
&&{}+2\lambda_i \int_H u(x)\langle
e_i,x \rangle\varphi(x) \gamma(dx).
\end{eqnarray*}

\textup{(ii)} The operator $\widetilde D\dvtx \mathcal D_0\subset L^2(H,\gamma)\to
L^2(H,\gamma;H)$ is closable.
\end{lemma}
\begin{pf}
(i) We have
%
\begin{eqnarray}
\label{e2.3} %
&& \int_H \bigl
\langle\widetilde D u(x),e_i \bigr\rangle\varphi(x) \gamma(dx)
\nonumber
\\[-8pt]
\\[-8pt]
\nonumber
&&\qquad=\lim_{s\to0}\frac{1}s \biggl[\int_H
u(x+se_i) \varphi(x) \gamma(dx)- \int_H u(x)
\varphi(x) \gamma(dx) \biggr].
\end{eqnarray}
But by the Cameron--Martin theorem the image measure $T_{se_i}(\gamma)$
of $\gamma$ under the translation $x\mapsto x+se_i$ is absolutely
continuous with respect to $\gamma$ with density (cf. \cite{D04}, Section~1.2.3)
\[
a_{se_i}(x)=e^{2s\lambda_i \langle e_i,x \rangle-s^2\lambda_i}.
\]
Hence, the difference of the two integrals on the right-hand side of
\eqref{e2.3} can be written as
\[
\int_H u(x)\bigl[\varphi(x-
\lambda_ie_i)-\varphi (x)\bigr]a_{se_i}(x)
\gamma(dx)+\int_H u(x)\varphi(x) \bigl(a_{se_i}(x)-1
\bigr) \gamma(dx).
\]
Hence, letting $s\to0$ in \eqref{e2.3} assertion (i) follows.

(ii) Suppose $u_n\in\mathcal D_0, n\in\mathbb N$, such that $
u_n\to0$ in $L^2(H,\gamma)$ and $\widetilde D u_n\to F$ in
$L^2(H,\gamma;H)$. Then for all $\varphi\in\mathcal FC^2_b, i\in
\mathbb N$, by (i)
\[
\int_H \bigl\langle F(x),e_i \bigr\rangle
\varphi(x) \gamma(dx)=\lim_{n\to\infty
}\int_H
\langle\widetilde D u_n,e_i \rangle\varphi(x)
\gamma(dx)=0.
\]
Hence, $F=0$.
\end{pf}

Let us denote the closure of $(\widetilde D,\mathcal D_0)$ again by
$\widetilde D$ and its domain by $\widetilde W ^{1,2}(H,\gamma)$.
Clearly, since $\mathcal FC^2_b\subset\mathcal D_0$ with $D\varphi
=\widetilde D\varphi$ for all $\varphi\in\mathcal FC^2_b$, it follows
that $ W ^{1,2}(H,\gamma)$ is a closed subspace of $\widetilde W
^{1,2}(H,\gamma)$. But in fact, they coincide.
%
\begin{lemma}
\label{l2.4}
$\mathcal FC^2_b$ is dense in $\widetilde W ^{1,2}(H,\gamma)$, hence
\[
W ^{1,2}(H,\gamma)=\widetilde W ^{1,2}(H,\gamma)
\]
and thus $Du=\widetilde Du$ for all $u\in W ^{1,2}(H,\gamma)$.
\end{lemma}

\begin{pf}
Let $u\in\widetilde W ^{1,2}(H,\gamma)$ such that
%
\begin{equation}
\label{e2.4} \int_H \langle\widetilde D \varphi,
\widetilde D u \rangle \,d\gamma+\int_H\varphi u \,d\gamma=0\qquad
\forall \varphi\in\mathcal FC^2_b.
\end{equation}
Since $\varphi(x)=\Phi(x_1,\ldots,x_N)$ for some $\Phi\in C^2_b(\mathbb
R^N)$ and $x_i:= \langle x,e_i \rangle$, $1\le i\le N$, we have that
\[
\langle\widetilde D \varphi, \widetilde D u \rangle=\sum
_{i=1}^N\langle D\varphi,e_i \rangle
\langle\widetilde D u,e_i \rangle
\]
and that $\langle D\varphi,e_i \rangle\in\mathcal FC^2_b$. Hence, by
Lemma~\ref{l2.3}, \eqref{e2.4} is equivalent to
%
\begin{equation}
\label{e2.5} -\int_H \bigl(2\mathcal L^{\mathrm{OU}}-1
\bigr)\varphi u \,d\gamma=0\qquad \forall \varphi\in\mathcal FC^2_b,
\end{equation}
where
\[
\mathcal L^{\mathrm{OU}}\varphi(x)=\tfrac{1}2 \operatorname{ Tr}
D^2\varphi(x)+ \langle x, ADx \rangle.
\]
But \eqref{e2.5} implies that $u=0$, since it is well known that
$\lambda-\mathcal L^{\mathrm{OU}}$ has dense range in $L^2(H,\gamma)$ for
$\lambda>0$. For the convenience of the reader we recall the argument:
The $C_0$-semigroup generated by the Friedrichs extension of the
symmetric operator $(\mathcal L^{\mathrm{OU}},\mathcal FC^2_b)$ on $L^2(H,\gamma
)$ is easily seen to be given by the following Mehler formula on
bounded, Borel functions $f\dvtx H\to\mathbb R$
%
\begin{equation}
\label{e8'} P_tf(x)=\int_H f
\bigl(e^{tA}x+y\bigr)N_{Q_t}(dy), \qquad t>0,
\end{equation}
where $N_{Q_t}$ is the centred Gaussian measure on $H$ with covariance operator
\[
Q_t:=\int_0^t
e^{2sA}\,ds,\qquad t>0.
\]
Obviously, $P_t(\mathcal FC^2_b)\subset\mathcal FC^2_b$, and also
\[
\biggl(\int_0^\infty e^{-\lambda t}
P_t \,dt \biggr) \bigl(\mathcal FC^2_b\bigr)
\subset\mathcal FC^2_b.
\]

But
\[
\bigl(\lambda-\mathcal L^{\mathrm{OU}} \bigr) ^{-1}=\int
_0^\infty e^{-\lambda
t} P_t \,dt
\]
as operators on $L^2(H,\gamma)$. Hence,
\[
\bigl(\lambda-\mathcal L^{\mathrm{OU}} \bigr) ^{-1}\bigl(\mathcal
FC^2_b\bigr)\subset \mathcal FC^2_b
\]
and so
\[
\mathcal FC^2_b\subset \bigl(1-\mathcal
L^{\mathrm{OU}} \bigr) \bigl(\mathcal FC^2_b\bigr).
\]
But $\mathcal FC^2_b$ is dense in $L^2(H,\gamma)$.
\end{pf}

Now we turn back to SDE \eqref{SDE}.

\begin{pf*}{Proof of Proposition~\ref{p2.1}} By (H2) and Lemma~\ref
{l2.2}, we have that $V\in\widetilde W ^{1,2}(H,\gamma)$ with $\nabla
V=\widetilde D V$, $\gamma$-a.e. Hence, Lemma~\ref{l2.4} implies the
assertion.
\end{pf*}

Let us consider the case when in SDE \eqref{SDE} we have that $B=0$,
that is,
%
\begin{eqnarray}
\label{e2.6} %
dX_t&=&
\bigl(AX_t-\nabla V(X_t)\bigr)\,dt+dW(t),
\nonumber
\\[-8pt]
\\[-8pt]
\nonumber
X_0&=&z,
\end{eqnarray}
where for convenience we extend $\nabla V\dvtx D_V\to H$ by zero to the
whole space $D_V$. The case for general $B$ then follows easily from
Girsanov's theorem.

To solve \eqref{e2.6} in the (probabilistically) weak sense, we shall
use \cite{AR91}, that is, the theory of Dirichlet forms, more precisely
the so-called ``classical (gradient type)'' Dirichlet forms, which for
the measure $\nu$ from the \hyperref[sec1]{Introduction} is just
\[
\mathcal E_\nu(u,v):=\int_H \bigl\langle
Du(x),Dv(x) \bigr\rangle \nu(dx),\qquad  u,v\in D(\mathcal E_\nu):=W^{1,2}(H,
\nu).
\]
But the whole theory has been developed for arbitrary finite measures
$m$ on $(H,\mathcal B(H))$ which satisfy an integration by parts
formula (see \cite{AR91,MR92} and the references therein) or
even more generally for finite measures $m$ for which $D\dvtx \mathcal
FC_b^\infty\subset L^2(H,m)\to L^2(H,m;H)$ is closable (see \cite
{AR89,AR90,MR92}). In particular, we can also take
$m:=\gamma$. Let us recall the following result which is crucial for
the theory of classical Dirichlet forms which we shall formulate for
$\nu$, but holds for every $m$ as above. For its formulation, we need
the notion of an ``$\mathcal E_\nu$-nest'': Let $F_n\subset H$, $n\in
\mathbb N$, be an increasing sequence of closed sets and define for
$n\in\mathbb N$
\[
D(\mathcal E_\nu)|_{F_n}:=\bigl\{u\in D(\mathcal
E_\nu)\dvtx u=0, \nu\mbox{-a.e. on } H\setminus F_n
\bigr\}.
\]
Then $(F_n)_{n\in\mathbb N}$ is called an \textit{$\mathcal E_\nu$-nest} if
\[
\bigcup_{k=1}^\infty D(\mathcal
E_\nu)|_{F_n} \qquad\mbox{is dense in } D(\mathcal
E_\nu),
\]
with respect to the norm
\[
\mathcal E_{\nu,1}(u,u)^{1/2}:= \bigl( \mathcal
E_\nu (u,u)+|u|^2_{L^2(H,\nu)} \bigr)^{1/2},\qquad
u\in D(\mathcal E_\nu),
\]
that is, with respect to the norm in $W^{1,2}(H,\nu)$.

Then the crucial result already mentioned is the following.
%
\begin{theorem}
\label{t2.5}
There exists an $\mathcal E_\nu$-nest consisting of compact sets.
\end{theorem}
\begin{pf}
See \cite{RS92} and \cite{MR92}, Chapter IV, Proposition~4.2.
\end{pf}

Let us denote $(K_n)_{n\in\mathbb N}$ this $\mathcal E_\nu$-nest
consisting of compacts. This theorem says that $(\mathcal E_\nu
,D(\mathcal E_\nu))$ is
completely determined in a $K_\sigma$ set of $H$. Then it follows from
the general theory that $(\mathcal E_\nu,D(\mathcal E_\nu))$ is \emph{quasi-regular}, hence has an associate Markov process which solves
(SDE) \eqref{e2.6} and this Markov process also lives on this
$K_\sigma$ set $\bigcup_{n=1}^\infty K_n$, that is, the first hitting
times $\sigma_{H\setminus K_n}$ of $H\setminus K_n$ converge to
infinity as $n\to\infty$.

The precise formulations of these facts is the contents of Theorems \ref
{t2.6} and~\ref{t2.7} below. We need one more notion: A set $N\subset
H$ is called $\mathcal E_\nu$-\emph{exceptional}, if it is contained in
the complement of an $\mathcal E_\nu$-nest. Clearly, this complement
has $\nu$-measure zero, hence $\nu(N)=0$ if $N\in\mathcal B(H)$.
%
\begin{theorem}
\label{t2.6}
There exists $S\in\mathcal B(H)$ such that $H\setminus S$ is $\mathcal
E_\nu$-exceptional [hence $\nu(H\setminus S)=0$] and for every $z\in
S$ there exists a probability space $(\Omega,\mathcal F,P_z)$ equipped
with a normal filtration $(\mathcal F_t)_{t\ge0}$, independent real
valued Brownian motions $W^k_t, t\ge0,  k\in\mathbb N$, on $(\Omega
,\mathcal F,P_z)$ and a continuous $H$-valued $(\mathcal F_t) $-adapted process $X_t, t\ge0$, such that $P_z$-a.s.:
\begin{longlist}[(iii)]
\item[(i)] $X_t\in S\  \forall t\ge0$,

\item[(ii)] $\int_H\mathbb E_{P_z} [ \int_0^t|\nabla
V(X_s)|^2\,ds  ]\nu(dz)<\infty$ and
$\mathbb E_{P_z} [ \int_0^t1_{H\setminus D_V}(X_s) \,ds
 ] =0 \forall t\ge0$,

\item[(iii)] $\langle e_k,X_t \rangle= \langle e_k,z
\rangle+\int_0^t( \langle Ae_k, X_s \rangle+ \langle e_k,\nabla V(X_s)
\rangle)\,ds+W^k_t, t\ge0,
k\in\mathbb N$.
\end{longlist}

Hence (by density), we have a solution of \eqref{e2.6} in the
sense of Definition~\ref{Def sol}. Furthermore, up to completing
$\mathcal F_t$ w.r.t. $P_z$, $(\Omega,\mathcal F)$, $X_t, t\ge0$,
and $(\mathcal F_t)$ can be taken canonical, independent of $z\in S$
and then
\[
\mathbf{M}:=\bigl(\Omega,\mathcal F, (\mathcal F_t)_{t\ge0},
(X_t)_{t\ge0},(P_z)_{z\in S}\bigr),
\]
forms a conservative Markov process, with invariant measure $\nu$.
\end{theorem}
\begin{pf}
The assertion follows from \cite{AR91}, Theorem~5.7.
\end{pf}

For later use, we define the Borel set
%
\begin{equation}
\label{e9'} H_V:= \biggl\{z\in H\dvtx \E_{P_z} \biggl[
\int_0^T\bigl|\nabla V(X_s)\bigr|^2\,ds
\biggr]<\infty \biggr\}
\end{equation}
and note that by Theorem~\ref{t2.6}(ii) we have $\nu(H_V)=\gamma(H_V)=1$.

In fact, by the convexity of $V$ we also have uniqueness for the
solutions to~\eqref{e2.6}. We recall that the sub-differential
$\partial V$ of $V$ is monotone (which is trivial to prove; see, e.g.,
\cite{P93}, Example~2.2a) and that for $x\in D_V$, $\partial
V(x)=\nabla V(x)$; see, for example, \cite{Ba10}, page~8. Hence, we have
%
\begin{equation}
\label{e2.6*} \bigl\langle\nabla V(x)-\nabla V(y),x-y \bigr\rangle\ge0,\qquad x,y\in
D_V.
\end{equation}
%
\begin{theorem}
\label{t2.6*}
Let $S$ be as in Theorem~\ref{t2.6} and $z\in S$. Then pathwise
uniqueness holds for all solutions in the sense of Definition~\ref{Def
sol} for SDE \eqref{e2.6}. In particular, uniqueness in law holds for
these solutions.
\end{theorem}
\begin{pf}
The first assertion is an immediate consequence of the monotonicity~\eqref{e2.6*}, since a part of our Definition~\ref{Def sol} requires
that the solutions are in $D_V$ $ dt\otimes P$-a.e.; see, for example,
\cite{DRW09}, proof of the claim page~1008/1009 or \cite{MR10}, Section~4,
for details. The second assertion then follows by the Yamada--Watanabe
theorem (see, e.g., \cite{RSZ08} which easily can be adapted to apply
to our case here).
\end{pf}
%
\begin{theorem}
\label{t2.7}
Let $\mathbf{M}$ be as in Theorem~\ref{t2.6} and let
$(F_n)_{n\in\mathbb N}$
be an $\mathcal E_\nu$-nest. Then
\[
P_z \Bigl[ \lim_{n\to\infty} \sigma_{H\setminus F_n} =
\infty \Bigr]=1,
\]
for all $z\in S\setminus N$, for some $\mathcal E_\nu$-exceptional set
$N$, where for a closed set $F\subset H$
\[
\sigma_{H\setminus F}:=\inf\{t>0\dvtx X_t\in H\setminus F\}
\]
is the first hitting time of $H\setminus F$.
\end{theorem}
\begin{pf}
Since $\mathbf{M}$ is conservative its lifetime $\zeta$ is
infinity. So, the assertion follows from \cite{MR92}, Chapter V, Proposition~5.30.
\end{pf}

Below we shall use the following simple lemma.
%
\begin{lemma}
\label{l2.8}
Let $(E,\|\cdot\|)$ be a Banach space and $V\dvtx E\to(-\infty,\infty]$ a
convex function.
\begin{longlist}[(ii)]
\item[(i)] Let $K\subset E$ be convex and compact such that $V(K)$ is
a bounded subset of~$\mathbb R$. Then the restriction of $V$ to $K$ is
Lipschitz.

\item[(ii)] Assume that $V$ is lower semi-continuous and $K\subset E$
compact such that $V(K)$ is an upper bounded subset of $\mathbb R$.
Then the restriction of $V$ to $K$ is Lipschitz.
\end{longlist}
\end{lemma}
\begin{pf}
(i) The proof is a simple modification of the classical proof that a
continuous convex function on an open subset of $E$ is locally
Lipschitz (see \cite{P93}, Proposition~1.6). For the convenience of the
reader, we give the argument:

 Define
\[
M=\sup_{x\in K}\bigl|V(x)\bigr|
\]
and
\[
d:=\operatorname{ diam} (K)\qquad
\bigl(:=\sup\bigl\{\|x-y\|\dvtx x,y\in K\bigr\}\bigr).
\]
Let $x,y\in K$. Set $\alpha:=\|x-y\|$ and
\[
z:=y+\frac{d}\alpha (y-x).
\]
Then $\|x-y\|\le d$, hence $z\in K$ since $K$ is convex. Furthermore,
\[
y=\frac{\alpha}{\alpha+d} z+\frac{d}{\alpha+d} x,
\]
hence
\[
f(y)\le\frac{\alpha}{\alpha+d} f(z)+\frac{d}{\alpha+d} f(x),
\]
so,
\[
f(y)-f(x)\le\frac{\alpha}{\alpha+d} \bigl(f(z)-f(x)\bigr)\le\frac{2M}{d} \| x-y
\|.
\]
Interchanging $x$ and $y$ in this argument, implies the
assertion.

(ii) This is an easy consequence of (i). Let $K_1$ be the
closed convex hull of~$K$. Then by Mazur's theorem $K_1$ is still
compact and by convexity $V(K_1)$ is an upper bounded subset of
$\mathbb R$. But $V(K_1)$ is also lower bounded, since $V$ is lower
semicontinuous. Hence, by (i) $V$ is Lipschitz on $K_1$, hence on $K$.
\end{pf}

Now let us come back to our convex function $V\dvtx H\to(-\infty,\infty]$
satisfying (H2) and (H3). We know by Proposition~\ref{p2.1} that $V\in
W^{1,2}(H,\nu)=D(\mathcal E_\nu)$.
Since $(\mathcal E_\nu,D(\mathcal E_\nu))$ is quasi-regular, it follows
by \cite{MR92}, Chapter IV, Proposition~3.3, that there exists an
$\mathcal E_\nu$-nest $(F_n)_{n\in\mathbb N}$ and a $\mathcal
B(H)$-measurable function $\widetilde V\dvtx H\to\mathbb R$ such that
%
\begin{equation}
\label{e2.7} \widetilde V=V \nu\mbox{-a.e. and } \widetilde
V|_{F_n} \mbox{ is continuous for every } n\in\mathbb N,
\end{equation}
where $\widetilde V|_{F_n}$ denotes the restriction of $\widetilde V$
to $F_n$. By
\cite{MR92}, Chapter III, Theorem~2.11, $(F_n\cap K_n)_{n\in\mathbb N}$
is again an $\mathcal E_\nu$-nest, where $( K_n)_{n\in\mathbb N}$ is
the $\mathcal E_\nu$-nest of compacts from Theorem~\ref{t2.5}. Since
$\nu(U)>0$ for every nonempty open set $U\subset H$, by \cite{MR92}, Chapter
III, Proposition~3.8, we can find an $\mathcal E_\nu$-nest
$( \widetilde F_n)_{n\in\mathbb N}$ such that $ \widetilde F_n\subset
F_n\cap K_n$ and the restriction of $\nu$ to $ \widetilde F_n$ has full
topological support on $ \widetilde F_n$ for every $n\in\mathbb N$,
that is, $\nu(U\cap\widetilde F_n)>0$ for all open $U\subset H$ with
$U\cap\widetilde F_n\neq\varnothing$. (Such an $\mathcal E_\nu$ set
is called \emph{regular}.) Since we want to fix this special regular
$\mathcal E_\nu$-nest of compacts depending on $V$ below, we assign to
it a special notation and set
%
\begin{equation}
\label{e2.8} K_n^V:=\widetilde F_n,\qquad n\in
\mathbb N.
\end{equation}
Now we can prove the following result which will be crucially used in
Section~\ref{sec6}.
%
\begin{proposition}
\label{p2.9}
\textup{(i)} Let $n\in\mathbb N$ and $K_n^V$ as in \eqref{e2.8}. Then
$V|_{K_n^V}$ is real valued, continuous and bounded. Furthermore,
$V(x)=\widetilde V(x)$ for every $x\in\bigcup_{n=1}^\infty
K_n^V$.

\textup{(ii) }There exists $S_V\in\mathcal B(H)$ such that $H\setminus
S_V$ is $\mathcal E_\nu$-exceptional, Theorem~\ref{t2.6} holds with
$S_V$ replacing $S$ and for every $z\in S_V$
%
\begin{equation}
\label{e13} P_z\Bigl[\lim_{n\to\infty}
\sigma_{H\setminus K^V_n}=\infty\Bigr]=1.
\end{equation}
\end{proposition}
\begin{pf}
(i) Since $ K^V_n\subset F_n$, we have for $\widetilde V$ from \eqref
{e2.7}, that $V|_{K^V_n}-\widetilde V|_{K^V_n}$ is lower
semi-continuous on $K^V_n$ with respect to the metric on $K^V_n$
induced by $|\cdot|$. Hence, $\{V|_{K^V_n}-\widetilde V|_{K^V_n}>0\}
=K^V_n\cap U$ for some open subset $U\subset H$. Since
$V|_{K^V_n}=\widetilde V|_{K^V_n}$ $\nu$-a.s., it follows, since
$(K_n^V)_{n\in\mathbb N}$ is a regular $\mathcal E_\nu$-nest that
\[
V(x)\le\widetilde V(x) \qquad\mbox{for every } x\in K^V_n.
\]
But $\widetilde V|_{K^V_n}$ is continuous, hence bounded, because
$K^V_n$ is compact, so $V(K^V_n)$ is an upper bounded subset of
$\mathbb R$, so by Lemma~\ref{l2.8}(ii) $V|_{K^V_n}$ is Lipschitz. But
then $\{V|_{K^V_n}\neq\widetilde V|_{K^V_n}\}=K^V_n\cap U$ for some
open subset $U\subset H$. Since $(K_n^V)_{n\in\mathbb N}$ is a regular
$\mathcal E_\nu$-nest, we conclude that
\[
V(x)=\widetilde V(x)\qquad \mbox{for every } x\in K^V_n.
\]
Hence, assertion (i) is proved.

(ii) By Theorem~\ref{t2.7}, we know that there exists an $\mathcal
E_\nu$-nest
$(F_n)_{n\in\mathbb N}$ such that
\[
P_z\Bigl[\lim_{n\to\infty}\sigma_{H\setminus K^V_n}=\infty
\Bigr]=1 \qquad\forall z\in\bigcup_{n=1}^\infty
F_n.
\]
Then by a standard procedure (see, e.g., \cite{MR92}, page~114) one can
construct the desired set $S_V\in\mathcal B(H)$.
\end{pf}

For the rest of this section, we fix $S_V$ as in Proposition~\ref{p2.9}.
%
\begin{corollary}
\label{c2.10}
Let $z\in S_V$ and $(X_t)_{t\ge0}$ a solution to \eqref{e2.6} on some
probability space $(\Omega, \mathcal F,P)$ with normal filtration and
cylindrical $(\mathcal F_t)$-Brownian motion $W=W_t, t\ge0$. Then
\eqref{e13} holds with $P$ replacing $P_z$.
\end{corollary}
\begin{pf}
This follows from the last part of Theorem~\ref{t2.6*}.
\end{pf}

It is now easy to prove existence of (probabilistic) weak solutions to
SDE \eqref{SDE} and uniqueness in law
%
\begin{theorem}
\label{t2.10}
For every $z\in S_V$, there exists a solution $Y=Y_t, t\in[0,T]$, to
SDE \eqref{SDE} on some probability space $(\Omega',\mathcal F',P')$ in
the sense of Definition~\ref{Def sol} and this solution is unique in
law. Furthermore, \eqref{e13} holds with $P'$ replacing $P_z$ and $Y$
replacing $X$, where $X=X_t, t\ge0$, is the process from Theorem~\ref
{t2.6} and if $z\in S_V\cap H_V$ [with $H_V$ as in \eqref{e9'}], then
%
\begin{equation}
\label{e14} \int_0^T\bigl|\nabla
V(Y_s)\bigr|^2\,ds<\infty\qquad  P'\mbox{-a.s.}
\end{equation}
\end{theorem}
\begin{pf}
This is now an easy consequence of Theorems \ref{t2.6},  \ref
{t2.6*} and Girsanov's theorem (see, e.g., \cite{DFPR13}, Appendix A1)
which easily extends to the present case since uniqueness in law holds
for SDE \eqref{e2.6}. To prove the last part, we note that by
Girsanov's theorem there exists a probability density $\rho\dvtx \Omega\to
(0,\infty)$ such that
\[
\bigl(\rho\cdot P'\bigr)\circ Y^{-1}=P_z
\circ X^{-1}.
\]
Hence, $P_z\circ X^{-1}=\rho_0 P'\circ Y^{-1}$, where $\rho_0$ is the
$P'\circ Y^{-1}$-a.s. unique function such that $\rho_0(Y)=\mathbb
E_{P'}[\rho|\sigma(Y)]$ $P$-a.s. and $\sigma(Y)$ denotes the $\sigma
$-algebra generated by $Y_t, t\in[0,T]$. So, \eqref{e13} and \eqref
{e14} follow, if $\rho_0>0  P'\circ Y^{-1}\mbox{-a.e.}$ To show
the latter, we first note that
\[
P'\circ Y^{-1}\bigl(\{\rho_0=0\}
\bigr)=P'\bigl(\bigl\{\mathbb E_{P'}\bigl[\rho|\sigma(Y)
\bigr]=0\bigr\}\bigr).
\]
But since
\[
\rho=e^{-\int_0^T\langle B(Y_s),dW_s\rangle-({1}/2)\int_0^T
|B(Y_s)|^2 \,ds}
\]
and $W$ is $\sigma(Y)$-measurable by SDE \eqref{SDE}, it follows that
$\mathbb E_{P'}[\rho|\sigma(Y)]=\rho$. But $\rho>0$.
\end{pf}

\section{Regularity theory for the corresponding Kolmogorov operator}\label{sec3}

\subsection{Uniform estimates on Lipschitz norms}\label{sec3.1}

First, we are concerned with the scalar equation
%
\begin{equation}
\lambda u-\mathcal{L}u-\bigl\langle B(x),Du\bigr\rangle=f, \label{e15}
\end{equation}
where $\lambda>0$, $f\in B_b(H)$ and $\mathcal{L}$ is the Kolmogorov operator
%
\begin{equation}
\mathcal{L}u(x)=\tfrac{1}{2} \operatorname{Tr} \bigl[D^{2}u(x)
\bigr]+\bigl\langle Ax-DV(x),Du(x)\bigr\rangle,\qquad  x\in H. \label{e16}
\end{equation}
Since the corresponding Dirichlet form
\[
\mathcal E_B(v,w):=\frac{1}2\int_H
\langle Dv,Dw \rangle \,d\nu-\int_H\langle B,Dv \rangle w\, d
\nu+\lambda\int_Hv w \,d\nu,
\]
$v,w\in W^{1,2}(H,\nu)$, is weakly sectorial for $\lambda$ big enough,
it follows by \cite{MR92}, Chapter~1 and Section~3e in Chapter II,
that \eqref{e15} has a unique solution $u\in L^2(H,\nu)$ such that
$u\in D(\mathcal{L})$. We need, however, Lipschitz regularity for $u$
and an estimate for its $\nu$-a.e. defined G\^ateaux derivative in
terms of $\|u\|_\infty$. To prove this, we also need the Kolmogorov
operator associated to the linear equation that one obtains, when
$B=V=0$, in SDE \eqref{SDE}, that is, the Ornstein--Uhlenbeck operator
%
\begin{equation}
\mathcal{L}^{\mathrm{OU}}u(x)=\tfrac{1}{2} \operatorname{Tr}
\bigl[D^{2}u(x)\bigr]+\bigl\langle Ax,Du(x)\bigr\rangle,\qquad x\in H.
\label{e17}
\end{equation}
As initial domains of $\mathcal{L}$, $\mathcal{L}^{\mathrm{OU}}$ and $\mathcal
{L}+ \langle B,D \rangle$, we take the set $\mathcal E_A(H)$ defined to
consist of the linear span of all real parts of functions $\varphi\dvtx H\to
\R$ of the form $\varphi(x)=e^{i\langle h,x \rangle}, x\in H$, with
$h\in D(A)$. It is easy to check that $\mathcal E_A(H)\subset
W^{1,2}(H,\gamma)$ densely and $\mathcal E_A(H)\subset W^{1,2}(H,\nu)$
densely. Then rewriting the last term in the above expression as $
\langle Ax,Du(x) \rangle$, the above operators are well defined for
$u\in\mathcal E_A(H)$. Below we are going to use results from \cite
{DR02} in a substantial way with $F:=\partial V$, the sub-differential
of $V$, which is maximal monotone (see, e.g., \cite{Ba10}) and which is
in general multi-valued, but single-valued on $D_V\subset D(F)$ because
$\partial V(x)=\nabla V(x)$ for $x\in D_V$.

Let us first check that assumptions (H1) and (H2) in there are satisfied.

First, Hypothesis 1.1 in \cite{DR02} is satisfied since we are in the
special case $A=A^*$ and $C=I$. Hypothesis 1.2(ii) is satisfied for
$\mathcal{L}$ defined above, replacing $N_0$ in \cite{DR02} with
$F_0:=\nabla V$, since by integrating by parts we have
\[
\int_H\mathcal{L} \varphi \psi \,d\nu=-\frac{1}2
\int_H \langle D\varphi, D\psi \rangle \,d\nu\qquad \forall
\varphi, \psi\in\mathcal E_A(H)
\]
and thus, taking $\psi=1$,
%
\begin{equation}
\label{e18'} \int_H\mathcal{L} \varphi \,d\nu=0\qquad \forall
\varphi, \psi\in \mathcal E_A(H).
\end{equation}
Here, $F_0$ is the minimal section of $F$ in \cite{DR02}, and hence
$\nabla V=F_0$ on $D(V)\subset D(F)$, so Hypothesis 1.2(iii) holds.
Hypothesis 1.2(i) follows from Remark~\ref{r0}.

The first result we now deduce from \cite{DR02} is the following.
%
\begin{proposition}
\label{p16}
$(\mathcal{L},\mathcal E_A(H))$ is closable on $L^2(H,\nu)$ and its
closure $(\mathcal{L},D(\mathcal L))$ is $m$-dissipative on $L^2(H,\nu)$.
\end{proposition}
\begin{pf}
This is a special case of \cite{DR02}, Theorem~2.3.
\end{pf}

For later use, we need to replace $\mathcal E_A(H)$ in Proposition~\ref
{p16} above by $\mathcal FC^2_b$ (defined in the \hyperref[sec1]{Introduction} of this paper).
We need the following easy lemma.
%
\begin{lemma}
\label{l17}
Let $\varphi\in C^2_b(\R^d)$. Then there exists a sequence $\varphi_n,
n\in\N$, each $\varphi_n$ consisting of linear combinations of
functions of type $x\to\cos\langle a,x \rangle_{\R^d}$, $a\in\R^d$,
such that $\sup_{n\in\N} \{\|\varphi_n\|_\infty+\|D\varphi_n\|
_\infty+\|D^2\varphi_n\|_\infty  \}<\infty$ and
\[
\lim_{n\to\infty}\varphi_n(x)=\varphi(x),\qquad \lim
_{n\to\infty}D\varphi _n(x)=D\varphi(x), \qquad\lim
_{n\to\infty}D^2\varphi_n(x)=D^2
\varphi(x),
\]
for all $x\in\R^d$.
\end{lemma}
\begin{pf}
First assume that $\varphi\in C^\infty_b(\R^d)$ with compact support.
Then we have
\[
\varphi(x)=\int_{\R^d}e^{i \langle x,\xi \rangle_{\R^d}} \widehat {\varphi}(
\xi) \,d\xi,\qquad x\in\R^d,
\]
where $\widehat{\varphi}$ is in the Schwartz test function space, with
the corresponding integral representations for $D\varphi$ and
$D^2\varphi$.

Discretizing the integrals immediately implies the assertion since
$x\mapsto(1+|x|^2)\widehat{\varphi}$ is Lebesgue integrable. Replacing
$\varphi$ by $\chi_n\varphi$ where $\chi_n, n\in\N$, is a suitable
sequence of localizing functions (\emph{bump functions}), the result
follows for all $\varphi\in C^\infty_b(\R^d)$ by regularization through
convolution with a Dirac sequence.
\end{pf}

As an immediate consequence of Proposition~\ref{p16} and Lemma~\ref
{l17}, we get the following.
%
\begin{proposition}
\label{p18}
$(\mathcal L,\mathcal FC^2_b)$ is closable on $L^2(H,\nu)$ and the
closure $(\mathcal L, D(\mathcal L))$ is the same as that in
Proposition~\ref{p16}, hence it is $m$-dissipative on $L^2(H,\nu)$.
Furthermore,
\[
\mathcal L u=\mathcal L^{\mathrm{OU}} u-\langle\nabla V, Du \rangle \qquad\forall u
\in\mathcal FC^\infty_b.
\]
\end{proposition}

Since $(\mathcal L, D(\mathcal L))$ is an $m$-dissipative operator on
$L^2(H,\nu)$ by Proposition~\ref{p16}, every $\lambda>0$ is in its
resolvent set, hence $(\lambda-\mathcal L)^{-1}$ exists as a bounded
operator on $L^2(H,\nu)$. The following is one of the main results in
\cite{DR02}.
%
\begin{theorem}
\label{t19}
Let $\lambda>0$ and $f\in B_b(H)$. Then there exists a $\nu$-version of
$(\lambda-\mathcal L)^{-1}f$ denoted by $R_\lambda f$, which is
Lipschitz on $H$, more precisely
%
\begin{equation}
\label{e18''} \bigl|R_\lambda f(x)-R_\lambda f(y)\bigr|\le\sqrt
\frac{\pi}{\lambda} \|f\| _\infty |x-y|\qquad \forall x,y\in H.
\end{equation}
\end{theorem}
\begin{pf}
We first notice that $H_0$, defined in \cite{DR02} to be the
topological support of~$\nu$, in our case is equal to $H$, since $\nu$
has the same zero sets as the (nondegenerate) Gaussian measure $\gamma
$ on $H$. Hence, the assertion follows from the last sentence of~\cite{DR02},
Proposition~5.2.
\end{pf}
%
\begin{remark}
\label{r20}
In fact, each $R_\lambda$ is a kernel of total mass
$\lambda^{-1}$, absolutely continuous with respect to $\nu$ and
$(R_\lambda)_{\lambda>0}$ forms a resolvent of kernels on $(H,\mathcal
B(H))$. We refer to \cite{DR02}, Section~5, for details.
\end{remark}

Now we are going to solve \eqref{e15} for each $f\in B_b(H)$ if $\lambda
$ is large enough, and show that the solution $u\in L^2(H,\nu)$ has a
$\nu$-version which is Lipschitz continuous, with Lipschitz constant
dominated up to a constant by $\|f\|_\infty$.

First, we need the following.
%
\begin{lemma}
\label{l21}
Let $g\dvtx H\to\R$ be Lipschitz. Then $g\in W^{1,2}(H,\gamma)$, hence also
in $W^{1,2}(H,\nu)$ and $\|Dg\|_\infty\le\|g\|_{\mathrm{Lip}}$ ($={}$Lipschitz norm
of $g$). Furthermore, $Dg=\nabla g$, $\gamma$-a.e. where $\nabla g$ is
the G\^ateaux derivative of $g$ which exists $\gamma$-a.e.
\end{lemma}
\begin{pf}
By the fundamental result in \cite{A76,P78} the set $D_g$ of
all $x\in H$ where $g$ is G\^ateaux-(even Fr\'echet-) differentiable
has $\gamma$ measure one. Let $\nabla g$ denote its G\^ateaux
derivative. Since $|\nabla g|\in L^\infty(H,\mu)$, it follows trivially
that $g\in\mathcal D_0$ defined in \eqref{e5'}. Hence, by Lemma~\ref
{l2.4} the assertion follows.
\end{pf}
%
\begin{lemma}
\label{l22}
Consider the operator $T_\lambda\dvtx L^\infty(H,\nu)\to L^\infty(H,\nu)$
defined by
\[
T_\lambda\varphi= \langle B,\nabla R_\lambda\varphi\rangle,\qquad
\varphi\in L^\infty(H,\nu).
\]
Then for $\lambda\ge4\pi\|B\|_\infty^2$
\[
\|T_\lambda\varphi\|_{L^\infty(H,\nu)}\le\tfrac{1}2 \|\varphi
\|_{L^\infty
(H,\nu)} \qquad\forall \varphi\in L^\infty(H,\nu).
\]
\end{lemma}
\begin{pf}
We have by \eqref{e18''} and Lemma~\ref{l21} that for $\varphi\in
L^\infty(H,\mu)$
\[
\|T_\lambda\varphi\|_{L^\infty(H,\nu)}\le\|B\|_\infty \sqrt
\frac{\pi
}{\lambda} \|\varphi\|_{L^\infty(H,\nu)},
\]
and the assertion follows.
\end{pf}
%
\begin{proposition}
\label{p23}
Let $f\in B_b(H)$ and $\lambda\ge4\pi\|B\|_\infty^2$. Then \eqref{e15}
has a unique solution given by the Lipschitz function
\[
u:=R_\lambda\bigl((I-T_\lambda)^{-1}f\bigr).
\]
This solution is Lipschitz on $H$ with Lipschitz norm
\[
\|u\|_{\mathrm{Lip}}\le2\sqrt\frac{\pi}{\lambda} \|f\|_\infty.
\]
\end{proposition}
\begin{pf}
Since the operator norm of $T_\lambda$ is less than $\frac{1}2$, the
operator $(I-T_\lambda)^{-1}$ exists as a continuous operator on
$L^\infty(H,\nu)$ with operator norm less than $2$. Furthermore, by
Theorem~\ref{t19} and Lemma~\ref{l21}
\begin{eqnarray*}
&&(\lambda-\mathcal L)R_\lambda
\bigl((I-T_\lambda)^{-1}f\bigr)- \bigl\langle
B,DR_\lambda\bigl((I-T_\lambda)^{-1}f\bigr) \bigr
\rangle
\\
&&\qquad=(I-T_\lambda)^{-1}f-T_\lambda\bigl((I-T_\lambda)^{-1}f
\bigr)=f.
\end{eqnarray*}
The final part follows from \eqref{e18''}
\end{pf}

Having established the result for the scalar equation (\ref{e15}) for
$\lambda\ge4\pi\|B\|_\infty^2$, we may
prove it for the vector equation (\ref{PDE vector}), whose solution $U$ has
components $u^{i}$ satisfying the equation
%
\begin{equation}
( \lambda+\lambda_{i} ) u^{i}-\mathcal{L}u^{i}-
\bigl\langle B(x),Du^{i}\bigr\rangle=f^{i}, \label{eq for U^i}
\end{equation}
where $f^{i}$ are the components of the vector function $F\dvtx H\rightarrow H$.
[$U ( x ) =\sum_{i=1}^{\infty}u^{i} ( x ) e_{i}$,
$F ( x ) =\sum_{i=1}^{\infty}f^{i} ( x ) e_{i}$].

We have by Proposition~\ref{p23}
\[
\bigl\llvert u^{i} ( x ) -u^{i} ( y ) \bigr\rrvert
^{2} \leq\frac{4\pi}{\lambda+\lambda_{i}}\bigl\Vert f^{i}\bigr\Vert_{\infty}^{2}
\llvert x-y\rrvert ^2 \leq\frac{4\pi}{\lambda+\lambda_{i}}\Vert F\Vert
_{\infty}^{2}\llvert x-y\rrvert ^2,
\]
hence
\[
\sum_{i=1}^{\infty}\bigl\llvert
u^{i} ( x ) -u^{i} ( y ) \bigr\rrvert ^{2}\leq c
( \lambda )^2 \Vert F\Vert_{\infty
}^{2}\llvert x-y
\rrvert ^2,
\]
where $c ( \lambda ): =\sum_{i=1}^{\infty}\frac{4\pi}{\lambda
+\lambda_{i}}$. This series converges and $\lim_{\lambda\rightarrow
\infty
}c ( \lambda ) =0$. Moreover, $\llvert  U ( x )
-U ( y ) \rrvert ^{2}=\sum_{i=1}^{\infty}\llvert
u^{i} ( x ) -u^{i} ( y ) \rrvert ^{2}$, hence we
have proved the following.
%
\begin{lemma}
\label{lemma uniform est}
$U (=U_\lambda)$ defined above satisfies
\[
\bigl\llvert U ( x ) -U ( y ) \bigr\rrvert \leq c ( \lambda ) \Vert F
\Vert_{\infty} \llvert x-y\rrvert,\qquad  x,y\in H
\]
with $\lim_{\lambda\rightarrow\infty}c ( \lambda ) =0$.
\end{lemma}
%
\subsection{It\^o formula for Lipschitz functions}\label{sec3.2}

Below we want to apply It\^o's formula to $u(X_t), t\ge0$, where $u$
is as in Proposition~\ref{p23} and $(X_t)_{t\ge0}$ are the paths of
the Markov process $\mathbf{M}$ from Theorem~\ref{t2.6}.
Since $u$ is only Lipschitz and we are on the infinite dimensional
state space $H$, this is a delicate issue. To give a technically clean
proof, we need a specific approximation of the solution $u$ in
Proposition~\ref{p23} by functions $u_n\in\mathcal FC^2_b$, $n\in\N$.
More precisely, we shall prove the following result.
%
\begin{proposition}
\label{p26new}
Let $\lambda>0$ and $g\in B_b(H)\cap D(\mathcal L^{\mathrm{OU}},C^1_{b,2}(H))$
(for the definition of the latter see below). Set
\[
w:=R_\lambda g.
\]
Then there exists a sequence $u_n\in\mathcal FC^2_b, n\in\N$, such that
%
\begin{eqnarray}
\label{e24new} %
 &&\sup_{n\in\N} \bigl(
\|u_n\|_\infty+\|\nabla u_n\|_\infty\bigr)\le 2
\sqrt\pi \max\bigl\{\lambda^{-1},\lambda^{-1/2}\bigr\} \|g
\|_\infty,
\nonumber
\\[-8pt]
\\[-8pt]
\nonumber
&&\lim_{n\to\infty}\int_H \bigl[\bigl|\mathcal
L^{\mathrm{OU}}(w-u_n)\bigr|^2+\bigl|\nabla (w-u_n)\bigr|^2+(w-u_n)^2
\bigr] \,d\nu=0.
\end{eqnarray}
In particular, $u_n\to w$ as $n\to\infty$ in $\mathcal L$-graph norm
[on $L^2(H,\nu)$] and
\[
\mathcal L=\mathcal L^{\mathrm{OU}}w- \langle\nabla V,\nabla w \rangle.
\]
\end{proposition}
For the proof, we need some more details from \cite{DR02}.

Define for $\lambda>0$ and $\varphi\in B_b(H)$
%
\begin{equation}
\label{e24'} R\bigl(\lambda,\mathcal L^{\mathrm{OU}}\bigr)\varphi(x)=\int
_0^\infty e^{-\lambda t} P_t
\varphi(x)\,dt,
\end{equation}
where $P_t$ is defined as in \eqref{e8'}. Then
\[
R\bigl(\lambda,\mathcal L^{\mathrm{OU}}\bigr) \bigl(C^1_{b,2}(H)
\bigr)\subset C^1_{b,2}(H),
\]
where $C^1_{b,2}(H)$ denotes the set of all $\varphi\in C^1_{b}(H)$
such that
\[
\sup_{x\in H} \frac{|\varphi(x)|}{1+|x|^2}<\infty \quad\mbox{and}\quad \sup
_{x\in H} \frac{|D\varphi(x)|}{1+|x|^2}<\infty.
\]
As in \cite{DR02}, we set
\[
D\bigl(\mathcal L^{\mathrm{OU}}, C^1_{b,2}(H)\bigr):=R
\bigl(\lambda,\mathcal L^{\mathrm{OU}}\bigr) \bigl(C^1_{b,2}(H)
\bigr),
\]
which by this resolvent equation is independent of $\lambda>0$ and is a
natural domain for the operator $\mathcal L^{\mathrm{OU}}$.
%
\begin{proposition}
\label{p17}
Let $u\in D(\mathcal L^{\mathrm{OU}}, C^1_{b,2}(H))$. Then there exists $\varphi
_n\in\mathcal E_A(H)$, $n\in\N$, such that $\varphi_n\to u$ in $\nu
$-measure and for some $C\in(0,\infty)$
\[
\bigl|\varphi_n(x)\bigr|+\bigl|D\varphi_n(x)\bigr|+\bigl|\mathcal
L^{\mathrm{OU}}\varphi_n(x)\bigr|\le C\bigl(1+|x|^2\bigr)\qquad
\forall x\in H, n\in\N.
\]
In particular, $u\in D(\mathcal L)$ and $\varphi_n\to u$ in the graph
norm of $\mathcal L$ on $L^2(H,\nu)$ and
\[
\mathcal L u=\mathcal L^{\mathrm{OU}}u- \langle\nabla V,Du \rangle.
\]
\end{proposition}
\begin{pf}
Since convergence in measure comes from a metrizable topology, this
follows from \cite{DR02}, Lemma~2.2, Lebesgue's dominatd convergence
theorem, Remark~0 and the fact that $(\mathcal L,\mathcal E_A(H))$ is
closable on $L^2(H,\nu)$.
\end{pf}

Now let us recall the approximation procedure for $\partial V$, more
precisely for its sub-differential $F:=-\partial V$ with domain $D(F)$,
performed in \cite{DR02}. [We recall that
$\nabla V$ is maximal monotone (see, e.g., \cite{Ba10}), hence we can
consider its Yosida approximations.]
For $\alpha\in(0,\infty)$,
we set
\[
F_\alpha(x):=\frac{1}\alpha \bigl(J_\alpha(x)-x\bigr),\qquad
x\in H,
\]
where
\[
J_\alpha(x):=(I-\alpha F)^{-1}(x),\qquad x\in H.
\]
It is well known (see, e.g., \cite{Ba10}) that
%
\begin{eqnarray}
\label{e18} %
\lim_{\alpha\to0}F_\alpha(x)&=&F_0(x)\qquad
\forall x\in D(F),
\nonumber
\\[-8pt]
\\[-8pt]
\nonumber
 \bigl|F_\alpha(x)\bigr|&\le& F_0(x) \qquad\forall x\in D(F),
\end{eqnarray}
where
\[
F_0(x):=\inf_{y\in F(x)} |y|.
\]
[Recall that $F(x)=\partial V(x)$ is in general multi-valued unless
$x\in D_V$, when $\partial V(x)=\nabla V(x)$.] We need a further
standard regularization by setting
%
\begin{eqnarray}
\label{e19} F_{\alpha,\beta}(x):=\int_H
e^{\beta B}F_\alpha\bigl(e^{\beta B}+y\bigr)
N_{({1}/2)B^{-1}(e^{2\beta B}-1)}(dy),
\nonumber
\\[-8pt]
\\[-8pt]
\eqntext{\alpha,\beta\in(0,\infty),}
\end{eqnarray}
where $B\dvtx D(B)\subset H\to H$ is a self-adjoint negative definite
operator such that $B^{-1}$ is of trace class. Then $F_{\alpha,\beta}$
is dissipative, of class $C^\infty$ has bounded derivatives of all
orders and
%
\begin{equation}
\label{e27new} \mbox{$F_{\alpha,\beta}\to F_\alpha$ pointwise as $
\beta\to0$ }
\end{equation}
(see \cite{DZ92}, Theorem~9.19).

Now let us fix $\lambda>0$ and consider the equation for $v\in C^2_b(H)$
%
\begin{equation}
\label{e20} \lambda u-\mathcal L^{\mathrm{OU}}u- \langle F_{\alpha,\beta},Du
\rangle=v.
\end{equation}
Then by \cite{DR02}, page~268, there exists a linear map
\[
R_\lambda^{\alpha,\beta}\dvtx C^2_b(H)\to D
\bigl(\mathcal L^{\mathrm{OU}}, C^1_{b,2}(H)\bigr)\cap
C^2_b(H)
\]
[in fact given by the resolvent of the SDE corresponding to the
Kolmogorov operator on the left-hand side of \eqref{e20}] such that
$R_\lambda^{\alpha,\beta}v$ is a solution to \eqref{e20} for each $v\in
C^2_b(H)$. In particular,
%
\begin{equation}
\label{e29new} \bigl\|R_\lambda^{\alpha,\beta} v\bigr\|_\infty\le
\frac{1}\lambda \|v\|_\infty ,\qquad  \lambda>0, v\in
C^2_b(H).
\end{equation}
We also have by \cite{DR02}, (4.7), that
%
\begin{equation}
\label{e30new} \sup_{x\in H}\bigl|\nabla R_\lambda^{\alpha,\beta}v(x)\bigr|
\le\sqrt{\frac{\pi}\lambda} \|v\|_\infty,\qquad \lambda>0, v\in
C^2_b(H).
\end{equation}
Now the proof of Proposition~\ref{p26new} will be the consequence of
the following two lemmas.
%
\begin{lemma}
\label{l28new}
Let $\alpha_n\in(0,\infty), n\in\N$, such that $\lim_{n\to\infty}\alpha
_n=0$. Then there exists $\beta_n\in(0,\infty), n\in\N$, such that for
all $v\in C^2_b(H)$ we have that
\[
\lim_{n\to\infty} R_\lambda^{\alpha_n,\beta_n}v=R_\lambda
v,
\]
in $\mathcal L$-graph norm [on $L^2(H,\nu)$].
\end{lemma}
\begin{pf*}{Proof \textup{(cf. the proof of \cite{DR02}, Theorem~2.3)}}
Since $D(F)\supset D_V$, so  $\nu(D(F))\ge\nu(D_V)=1$, it follows by
\eqref{e18} and Lebesgue's dominated convergence theorem that
%
\begin{equation}
\label{e31new} \lim_{n\to\infty}\int_H|F_{\alpha_n}-
\nabla V|^2\,d\nu=0.
\end{equation}
Since by the definition of $F_{\alpha,\beta}$ we have that for each
$\alpha>0$ there exists $c_\alpha\in(0,\infty)$ such that
\[
\bigl|F_{\alpha,\beta}(x)\bigr|\le c_\alpha\bigl(1+|x|\bigr)\qquad\forall x\in H,
\]
it follows by \eqref{e27new} that for $n\in\N$ there exists $\beta_n\in
(0,\tfrac{1}n)$, such that
\[
\int_H\bigl|F_{\alpha_n,\beta_n}(x)-F_{\alpha_n}(x)\bigr| \nu(dx)
\le\frac{1}n.
\]
Hence, by \eqref{e31new}
%
\begin{equation}
\label{e32new} \lim_{n\to\infty}\int_H|F_{\alpha_n,\beta_n}-
\nabla V|^2\,d\nu=0.
\end{equation}
Now let $v\in C^2_b(H)$. Then $R_\lambda^{\alpha_n,\beta_n}v\in
C^2_b(H)\cap D(\mathcal L^{\mathrm{OU}}, C^1_{b,2}(H)) $, hence by Proposition~\ref{p17} and, because $R_\lambda^{\alpha_n,\beta_n}v$ solves \eqref
{e20}, we have
%
\begin{equation}
\label{e33new} (\lambda-\mathcal L)R_\lambda^{\alpha_n,\beta_n}v=v+ \bigl
\langle F_{\alpha
_n,\beta_n}-\nabla V,\nabla R_\lambda^{\alpha_n,\beta_n}v \bigr
\rangle ,
\end{equation}
consequently,
%
\begin{equation}
\label{e34new} R_\lambda^{\alpha_n,\beta_n}v=(\lambda-\mathcal
L)^{-1}v+ (\lambda -\mathcal L)^{-1} \bigl(\bigl\langle
F_{\alpha_n,\beta_n}-\nabla V,\nabla R_\lambda^{\alpha_n,\beta_n}v \bigr\rangle
\bigr).
\end{equation}
But by \eqref{e30new} and \eqref{e32new}
\[
\lim_{n\to\infty}\int_H\bigl|\bigl\langle
F_{\alpha_n,\beta_n}-\nabla V,\nabla R_\lambda^{\alpha_n,\beta_n}v \bigr
\rangle\bigr|^2 \,d\nu=0.
\]
Hence, \eqref{e33new} and \eqref{e34new} imply the assertion, because
$(\lambda-\mathcal L)^{-1}$ is continuous on $L^2(H,\nu)$ ad $R_\lambda
v$ is a $\nu$-version of $(\lambda-\mathcal L)^{-1}v$.
\end{pf*}
%

\begin{lemma}
\label{l29new}
Let $\lambda,g$ and $w$ be as in Proposition~\ref{p26new}. Then there
exist $u_n\in C^2_b(H)\cap D(\mathcal L^{\mathrm{OU}}, C^1_{b,2}(H)), n\in\N$,
such that
\[
\sup_{n\in\N}\bigl(\|u_n\|_\infty+\|\nabla
u_n\|_\infty\bigr)\le2\sqrt\pi \max\bigl\{ \lambda^{-1},
\lambda^{-1/2}\bigr\} \|g\|_\infty,
\]
and \eqref{e24new} holds for these $u_n, n\in\N$.
\end{lemma}
\begin{pf}
Since $C^2_b(H)\subset L^2(H,\nu)$ densely, we can find $v_k\in
C^2_b(H), k\in\N$, such that
\[
\sup_{k\in\N}\|v_k\|_\infty\le2\|g
\|_\infty
\]
and
\[
\lim_{k\to\infty}\int_H|g-v_k|^2
\,d\nu=0,
\]
hence by the continuity of $(\lambda-\mathcal L)^{-1}$,
\[
\lim_{k\to\infty}\int_H\bigl|R_\lambda(g-v_k)\bigr|^2
\,d\nu=0.
\]
Therefore, $R_\lambda v_k\to R_\lambda g$ in $\mathcal L$-graph norm
[on $L^2(H,\nu)$] as $k\to\infty$. Hence, by Lemma~\ref{l28new} we can
choose a subsequence $(k_n)_{n\in\N}$ such that
\[
R_\lambda^{\alpha_{n_k},\beta_{n_k}}v_k\to R_\lambda g\qquad
\mbox{in }\mathcal L\mbox{-graph norm} \bigl(\mbox{on }L^2(H,\nu)\bigr)
\mbox{ as }k\to
\infty.
\]
Taking $u_k:=R_\lambda^{\alpha_{n_k},\beta_{n_k}}v_k, k\in\N$, the
assertion follows from \eqref{e29new} and \eqref{e30new}, recalling
that convergence in $\mathcal L$-graph norm implies convergence in
$W^{1,2}(H,\nu)$.
\end{pf}

\begin{pf*}{Proof of Proposition~\ref{p26new}} Let $u\in
C^2_b(H)\cap D(\mathcal L^{\mathrm{OU}},C^1_{b,2}(H))$ and define $u_n:=u\circ
P_n\in\mathcal FC^2_b, n\in\N$. Then $\|u_n\|_\infty\le\|u\|_\infty$
and $\|\nabla u_n\|_\infty\le\|\nabla u\|_\infty$. Furthermore,
$u_n\to u$, $\nabla u_n\to\nabla u$ and $\mathcal L^{\mathrm{OU}}u_n\to
\mathcal L^{\mathrm{OU}} u$ pointwise on $H$ as \mbox{$n\to\infty$}. Furthermore,
$\mathcal L^{\mathrm{OU}} u_n\to\mathcal L^{\mathrm{OU}} u$ in $L^2(H,\gamma)$, hence in
$L^2(H,\nu)$ as $n\to\infty$. Now the assertion follows by Lemma~\ref
{l29new}.
\end{pf*}

%
\begin{corollary}
\label{c30new}
Let $f\in B_b(H),
\lambda\ge4\pi\|B\|^2_\infty$ and $u$ as in Proposition~\ref{p23},
that is,
\[
u:=R_\lambda\bigl((I-T_\lambda)^{-1}f\bigr).
\]
Let $u_n\in\mathcal F C^2_b\cap D(\mathcal L^{\mathrm{OU}},C^1_{b,2}(H)), n\in
\N$, be as in Proposition~\ref{p26new} with $g:=(I-T_\lambda)^{-1}f$
[$\in B_b(H)$, with $\|g\|_\infty\le2\|f\|_\infty$ by the proof of
Proposition~\ref{p23}].
Consider the Markov process
\[
\mathbf{M}:=\bigl(\Omega,\mathcal F, (\mathcal F_t)_{t\ge
0},(X_t)_{t\ge0},
(P_z)_{z\in S_V}\bigr)
\]
from Theorem~\ref{t2.6}, with $S_V$ defined in Proposition~\ref{p2.9}.
Then there exists an $\mathcal E_\nu$-nest $(F_k^{\lambda,f})_{k\in\N
}$ of compacts such that for every $k\in\N$, $F_k^{\lambda,f}\subset
S_V$ and some subsequence $n_l\to\infty$:
\begin{longlist}[(iii)]
\item[(i)]$u_{n_l}(z)\to u(z)$,

\item[(ii)] $\E_{P_z}\int_0^\infty e^{-\lambda s}|\nabla u-\nabla
u_{n_l}|^2(X_s)\,ds=R_\lambda(|\nabla u-\nabla u_{n_l}|^2)(z)\to0$,

\item[(iii)] $\E_{P_z}\int_0^\infty e^{-s}|\mathcal
L(u-u_{n_l})(X_s)|\,ds\to0$,

uniformly in $z\in F_k^{\lambda,f}$. In particular, for all
$z\in\bigcup_{k=1}^\infty F_k^{\lambda,f}\setminus N$ with an $\mathcal
E_\nu$-exceptional set $N$, we have that $P_z$-a.e. the following It\^
o formula holds:
%
\begin{equation}
\label{e35new} u(X_t)-z-\int_0^t
\mathcal L u(X_s)\,ds=\int_0^t
\bigl\langle \nabla u(X_s),dW(s)\bigr\rangle\qquad \forall t\ge0.
\end{equation}
\end{longlist}
\end{corollary}
\begin{pf}
Since the convergence of all three sequences in (i)--(iii) takes place
in $W^{1,2}(H,\nu)$, the existence of such an $\mathcal E_\nu$-nest and
subsequence $(n_l)_{l\in\N}$ follows from \cite{MR92}, Chapter III,
Proposition~3.5, and Theorem~\ref{t2.5} above. By Theorem~\ref
{t2.7} for $z\in\bigcup_{k=1}^\infty F_k^{\lambda,f}\setminus N$, for
some $\mathcal E_\nu$-exceptional set $N$ we know that $P_z[\bigcup_{k=1}^\infty\{\tau_{H\setminus F_k^{\lambda,f}}>t\}]=1$ for all $t\ge
0$. So, fix $z\in\bigcup_{k=1}^\infty F_k^{\lambda,f}\setminus N$. Then
by the classical It\^o formula on finite dimensional Euclidean space
and by Theorem~\ref{t2.6}(iii), we have $P_z$-a.s.
%
\begin{eqnarray}
\label{e36new} %
&& u_{n_l}(X_t)-z-
\int_0^t\bigl(\mathcal L^{\mathrm{OU}}
u_{n_l}-\langle\nabla V,\nabla u_{n_l} \rangle\bigr)
(X_s)\,ds
\nonumber
\\[-8pt]
\\[-8pt]
\nonumber
&&\qquad=\int_0^t \bigl\langle \nabla
u_{n_l}(X_s),dW(s)\bigr\rangle \qquad\forall t\ge0.
\end{eqnarray}
Fix $t>0$. Then on $ \{\tau_{H\setminus F_k^{\lambda,f}}>t\}$ we have
by (ii) above that $u_{n_l}(X_t)\to u(X_t)$ as $n\to\infty$ and by the
last part of Proposition~\ref{p18} and (iii) above
\begin{eqnarray*}
&&E_{P_z}\int_0^t\bigl|
\bigl(\mathcal Lu-\bigl(\mathcal L^{\mathrm{OU}} u_{n_l}-\langle \nabla
V,\nabla u_{n_l} \rangle\bigr) (X_s)\bigr)\bigr|\,ds
\nonumber
\\[-8pt]
\\[-8pt]
\nonumber
&&\qquad\le e^{ t} \E_{P_z}\int_0^\infty
e^{-s}\bigl|\mathcal L(u-u_{n_l}) (X_s)\bigr|\,ds\to0\qquad
\mbox{as } l\to\infty,
\end{eqnarray*}
and also that by It\^o's isometry and by (ii) above
\begin{eqnarray*}
&&\E_{P_z}\biggl\llvert \int
_0^t\bigl\langle\nabla u(X_s)-
\nabla u_{n_l}(X_s), dW_s \bigr\rangle \biggr
\rrvert ^2
\\
&&\qquad\le\E_{P_z}\int_0^t\bigl | \nabla
u(X_s)-\nabla u_{n_l}(X_s)\bigr|^2\,ds
\\
&&\qquad\le e^{\lambda t}\int_0^\infty
e^{-\lambda s} \E_{P_z} \bigl( \bigl| \nabla u(X_s)-\nabla
u_{n_l}(X_s)\bigr|^2 \bigr) \,ds
\\
&&\qquad=e^{\lambda t}R_\lambda\bigl(| \nabla u -\nabla u_{n_l}
|^2\bigr) (z)\to0\qquad \mbox{as } l\to\infty.
\end{eqnarray*}
Hence, on $\bigcup_{k=1}^\infty\{\tau_{H\setminus F_k^{\lambda,f}}>t\}$
we can pass to the limit in \eqref{e36new} to get \eqref{e35new}.
\end{pf}
%
\begin{remark}
\label{r31new}
By the same standard procedure already mentioned at the end of the
proof of Proposition~\ref{p2.9}, we can find $S^{\lambda,f}_V$ such
that $H\setminus S_V^{\lambda,f}$ is $\mathcal E_\nu$-exceptional and
Theorem~\ref{t2.6},
Proposition~\ref{p2.9}, Theorem~\ref{t2.10} hold with $S^{\lambda,f}_V$
replacing $S_V$ and for all $z\in S^{\lambda,f}_V$, (i)--(iii) in
Corollary~\ref{c30new} hold and \eqref{e35new} holds $P_z$-a.s.
\end{remark}

\subsection{Maximal regularity estimates}\label{sec3.3}

Let us first consider again the solution $u$ of the scalar equation
(\ref{e15}). The following
result is the main technical ingredient of this paper, on the Kolmogorov
equation; see \cite{DL14}, Proposition~4.2.

\begin{lemma}
\label{l32'}
We have that $u\in W^{2,2}(H,\nu)$ and there is a constant $C>0$ such
that, for all $\lambda\geq1$,
%
\begin{eqnarray}
\label{e36'} \int_{H}\bigl\llvert Du ( x ) \bigr\rrvert
^{2}\nu ( dx ) &\leq&\frac{C}{\lambda}\int_{H}
\bigl\llvert f ( x ) \bigr\rrvert ^{2}%
\nu ( dx ),
\\
\label{e36''} \int_{H}\bigl\llVert D^{2}u ( x
) \bigr\rrVert _{HS}^{2}\nu ( dx )& \leq & C\int
_{H}\bigl\llvert f ( x ) \bigr\rrvert ^{2}%
\nu ( dx ).
\end{eqnarray}
\end{lemma}

We then apply this result componentwise to equation (\ref{PDE vector}).

\begin{theorem}
\label{t33'}Let $U ( x ) =\sum_{i=1}^{\infty}%
u^{i} ( x ) e_{i}$ be the solution of equation (\ref{PDE vector})
with $F ( x ) =\sum_{i=1}^{\infty}f^{i} ( x ) e_{i}$,
namely $u=u^{i}$ satisfies equation (\ref{eq for U^i}) with $f=f^{i}$, for
every $i\in\mathbb{N}$. Then
\[
\int_{H}\sum_{i=1}^{\infty}
\bigl( \lambda_{i}\bigl\llvert Du^{i} ( x ) \bigr\rrvert
^{2}+\bigl\llVert D^{2}u^{i} ( x ) \bigr\rrVert
_{HS}^{2} \bigr) \nu ( dx ) \leq C\int_{H}
\bigl(\bigl\llvert F ( x ) \bigr\rrvert ^{2}+\bigl|B(x)\bigr|^2\bigr)
\nu ( dx ).
\]
\end{theorem}

\begin{pf}
We apply the lemma and get
\begin{eqnarray*}\int_{H}\bigl\llvert
Du^{i} ( x ) \bigr\rrvert ^{2}\nu ( dx ) &\leq&
\frac{C}{\lambda+\lambda_{i}}\int_{H}\bigl(\bigl\llvert
f^{i} ( x ) \bigr\rrvert ^{2}+\bigl|\bigl\langle
B(x),e_i \bigr\rangle\bigr|^2\bigr)\nu ( dx )
\\
& \leq&\frac{C}{\lambda_{i}}%
\int_{H}\bigl(\bigl
\llvert f^{i} ( x ) \bigr\rrvert ^{2}+\bigl|\bigl\langle
B(x),e_i \bigr\rangle\bigr|^2\bigr)\nu ( dx ),\\
\int_{H}\bigl\llVert D^{2}u^{i} ( x
) \bigr\rrVert _{HS}^{2}\nu ( dx ) &\leq & C\int
_{H}\bigl(\bigl\llvert f^{i} ( x ) \bigr\rrvert
^{2}+\bigl|\bigl\langle B(x),e_i \bigr\rangle\bigr|^2
\bigr) \nu ( dx ).
\end{eqnarray*}
Therefore,
\begin{eqnarray*}
&& \int_{H}\sum_{i=1}^{\infty}
\bigl( \lambda_{i}\bigl\llvert Du^{i} ( x ) \bigr\rrvert
^{2}+\bigl\llVert D^{2}u^{i} ( x ) \bigr\rrVert
_{HS}^{2} \bigr) \nu ( dx )
\\
&&\qquad \leq2C\int_{H}\bigl(\bigl\llvert F ( x ) \bigr\rrvert
^{2}+\bigl|B(x)\bigr|^2\bigr)\nu ( dx ) <\infty.
\end{eqnarray*}
The proof is complete.
\end{pf}
%
\begin{remark}
\label{r33'}
Consider the situation of Lemma~\ref{l32'} and let $(u_l)_{l\in\N}$
be the sequence $(u_{n_l})_{l\in\N}$ from Corollary~\ref{c30new}. Then
it follows by Proposition~\ref{p26new} and Corollary~\ref{c30new} that
as $n\to\infty$
\[
f_n:=(\lambda-\mathcal L)u_n+\langle B,Du_n
\rangle\to f\qquad \mbox{in } L^2(H,\nu).
\]
Hence, by \eqref{e36''},
\[
\lim_{n\to\infty}\int_H\bigl\|D^2(u-u_n)
\bigr\|^2_{HS} \,d\nu=0.
\]
This will be crucially used to justify the application of mean value
theorem in the proof of Lemma~\ref{l38} below.
\end{remark}

\section{New formulation of the SDE}\label{sec4}
In this section, we fix $U$, $u^i$ as in Theorem~\ref{t33'} with
$f^i:=\langle B,e_i \rangle$ and $F:=B$. Let $\lambda\ge4\pi\|B\|
^2_\infty$ so large that $c(\lambda)\le\frac{1}2 \|B\|^{-1}_\infty$
where $c(\lambda)$ is as in Lemma~\ref{lemma uniform est}. Again, we write $x^i$ for
$\langle x,e_i\rangle$, $u^i(x)$ for $\langle U(x),e_i\rangle$, and so
on. Below we shall apply Corollary~\ref{c30new} with $f$ replaced by
$B^i$ and $u^i$ replacing $u$, for $i\in\N$.
%
\begin{remark}
\label{r34'}
As the corresponding sets of allowed starting points
$S_V^{B^i,\lambda}, i\in\N$, are concerned, as in Remark~\ref{r31new},
by a standard diagonal procedure we can find $S_V\subset\bigcap_{i\in\N
}S_V^{B^i,\lambda}$ such that $H\setminus S_V$ is $\mathcal E_\nu
$-exceptional and Theorem~\ref{t2.6}, Proposition~\ref{p2.9}, Theorem~\ref{t2.10} hold with this (smaller) $S_V$ and for all $z\in S_V$
(i)--(iii) in Corollary~\ref{c30new} hold and \eqref{e35new} holds
$P_z$-a.s.

Below we fix this set $S_V(\subset H)$.
\end{remark}

\begin{lemma}
\label{l36}Let $z\in S_V$ and set
\[
\varphi ( x ) =x+U ( x ),\qquad  x\in H,
\]
namely $\varphi^{i} ( x ) =x^{i}+u^{i} ( x ) $ and let
$X$ be a solution of the SDE (\ref{SDE}). Then for each $i\in\N$
%
\begin{eqnarray}
\label{e37'} d\varphi^{i} ( X_{t} ) &= &\bigl( -\lambda
_{i}X_{t}^{i}-D_{i}V (
X_{t} ) \bigr) \,dt+ ( \lambda+\lambda_{i} ) u^{i}
( X_{t} ) \,dt
\nonumber
\\[-8pt]
\\[-8pt]
\nonumber
&&{}+ \bigl\langle Du^{i} ( X_{t} )
,dW_{t} \bigr\rangle +dW_{t}^{i}.
\end{eqnarray}
\end{lemma}
\begin{pf}
Fix $i\in\N$. Let us first prove the following.

\textit{Claim}: We have $P_z$-a.e.
%
\begin{eqnarray}
\label{e38'} %
 u^i(X_t)&=&
 u^i(z)+\int_0^t\bigl(\mathcal
L^{\mathrm{OU}} u^i(X_s)-\bigl\langle\nabla
V(X_s)-B(X_s), D u(X_s) \bigr\rangle\bigr)
\,ds
\nonumber
\\[-8pt]
\\[-8pt]
\nonumber
&&{}+\int_0^t\bigl\langle D
u^i(X_s),dW_s, \bigr\rangle \,ds, \qquad t\ge0.
\end{eqnarray}
Indeed, considering the set $\Omega_0$ of all $\omega\in\Omega$ such
that \eqref{e38'} holds, we have to prove that $P(\Omega_0)=1$. But by
Girsanov's theorem this is equivalent to \eqref{e35new} with $u^i$
replacing $u$. Hence, the claim is proved.

As a consequence, we obtain that
\begin{eqnarray*}
du^{i} ( X_{t} ) & =&\mathcal{L}u^{i} (
X_{t} ) \,dt+B(X_t)\,dt+ \bigl\langle Du^{i} (
X_{t} ) ,dW_{t} \bigr\rangle
\\
& =&\mathcal{-}B^{i} ( X_{t} ) \,dt+ ( \lambda+\lambda
_{i} ) u^{i} ( X_{t} ) \,dt+ \bigl\langle
Du^{i} ( X_{t} ) ,dW_{t} \bigr\rangle
\end{eqnarray*}
and thus
\begin{eqnarray*} dX_{t}^{i}&= &\bigl( -
\lambda_{i}X_{t}^{i}-D_{i}V (
X_{t} ) \bigr) \,dt-du^{i} ( X_{t} )
\\
&&{}+ ( \lambda+\lambda_{i} ) u^{i} ( X_{t} ) \,dt+
\bigl\langle Du^{i} ( X_{t} ) ,dW_{t} \bigr
\rangle+dW_{t}^{i}.
\end{eqnarray*}
Then
\begin{eqnarray*}
d\varphi^{i} ( X_{t} ) &= &\bigl( -\lambda
_{i}X_{t}^{i}-D_{i}V (
X_{t} ) \bigr) \,dt+ ( \lambda+\lambda_{i} ) u^{i}
( X_{t} ) \,dt\\
&&{}+ \bigl\langle Du^{i} ( X_{t} )
,dW_{t} \bigr\rangle +dW_{t}^{i}.
\end{eqnarray*}
\upqed\end{pf}

In vector form, we could write \eqref{e37'} as
\[
dX_{t}= \bigl( AX_{t}-\nabla V ( X_{t} ) \bigr)
\,dt-dU ( X_{t} ) + ( \lambda-A ) U ( X_{t} ) \,dt+DU (
X_{t} ) \,dW_{t}+dW_{t}%
\]
and
\[
d\varphi ( X_{t} ) = \bigl( AX_{t}-\nabla V (
X_{t} ) \bigr) \,dt+ ( \lambda-A ) U ( X_{t} ) \,dt+DU (
X_{t} ) \,dW_{t}+dW_{t}.
\]

\section{Proof of Theorem \texorpdfstring{\protect\ref{t4}}{1.5}}\label{sect proof uniqueness}

Consider the situation described at the beginning of Section~\ref{sec4} with
$S_V$ being the set of all allowed starting points from Remark~\ref{r34'}.
In particular, by our choice of $\lambda$ we have
\[
\sup_{x\in H}\bigl\llVert \nabla U ( x ) \bigr\rrVert
_{\mathcal
{L} (
H ) }\leq\tfrac{1}{2}.
\]

\begin{lemma}
\label{l37}For every $x,y\in H$, we have
\[
\tfrac{1}2 \llvert x-y\rrvert \leq\bigl\llvert \varphi ( x ) -\varphi (
y ) \bigr\rrvert \le\tfrac{3}2 |x-y|.
\]
In particular, $\varphi$ is injective and its inverse is Lipschitz continuous.
\end{lemma}

\begin{pf}
One has
\begin{eqnarray*}
\llvert x-y\rrvert & \leq&\bigl\llvert x+U ( x ) -y-U ( y ) \bigr\rrvert +\bigl
\llvert U ( x ) -U ( y ) \bigr\rrvert
\\
& \leq&\bigl\llvert \varphi ( x ) -\varphi ( y ) \bigr\rrvert +\tfrac{1}{2}
\llvert x-y\rrvert ,
\end{eqnarray*}
where we have used
\[
\bigl\llvert U ( x ) -U ( y ) \bigr\rrvert \leq\sup_{x\in
H}
\bigl\llVert DU ( x ) \bigr\rrVert \llvert x-y\rrvert \leq\tfrac{1}{2}\llvert
x-y\rrvert .
\]
The claim follows.
\end{pf}

Let $X$ and $Y$ be two solutions with initial condition $x$, defined on the
same filtered probability space $ ( \Omega,\mathcal{F}, (
\mathcal{F}_{t} ) _{t\geq0},P ) $ and w.r.t. the same cylindrical
$(\mathcal{F}_{t})$-Brownian motion $W$.

\begin{lemma}
\label{l38}There is a Borel set $\Xi\subset S_V$ with $\gamma (
\Xi ) =1$ having the following property: If $z\in\Xi$ and $X$, $Y$ are
two solutions with initial condition $z$ (in the sense of Definition~\ref{Def sol}), defined on the same filtered probability space $ (
\Omega,\mathcal{F}, ( \mathcal{F}_{t} ) _{t\geq0},P ) $ and
w.r.t. the same $(\mathcal{F}_{t})$-cylindrical Brownian motion $W$,
then
\[
A_{t,z}<\infty
\]
with probability one, for every $t\geq0$, where the process $A_{t,z}$ is
defined as
%
\begin{eqnarray}
\label{e40'} %
A_{t,z}&=&2\int
_0^t\frac{|\nabla V(X_s)-\nabla V(Y_s)|}{|\varphi
(X_s)-\varphi(Y_s)|} 1_{\varphi(X_s)\neq\varphi(Y_s)} \,ds
\nonumber\\
&&{}+2\sum_{i=1}^\infty\int
_0^t\frac{(u^i(X_s)-u^i(Y_s))^2}{|\varphi
(X_s)-\varphi(Y_s)|^2} 1_{\varphi(X_s)\neq\varphi(Y_s)} \,ds
\\
&&{}+\sum_{i=1}^\infty\int
_0^t\frac{|Du^i(X_s)-Du^i(Y_s)|^2}{|\varphi
(X_s)-\varphi(Y_s)|^2} 1_{\varphi(X_s)\neq\varphi(Y_s)} \,ds.\nonumber
\end{eqnarray}
\end{lemma}

\begin{pf}
Let us first treat the case when (H3) holds.
By the mean value theorem and Lemma~\ref{l37}, we have for $\nu$-a.e. $z\in S_V$
\[
A_{t}\leq4N_{t,z},
\]
where
\begin{eqnarray*}N_{t,z}:&=&2\int
_{0}^{1}\int_{0}^{t}
\bigl\|D^2V\bigl(Z^\alpha_s\bigr)\bigr\|_{\mathcal
L(H)} \,d
\alpha\, ds
\\
&&{}+\sum_{i=1}^\infty\int
_{0}^{1}\int_{0}^{t}
\bigl(2\lambda_i\bigl|Du^i\bigl(Z^\alpha_s
\bigr)\bigr|^2+\bigl \|D^2u^i\bigl(Z^\alpha_s
\bigr)\bigr\|_{HS}^2 \bigr) \,d\alpha \,ds,
\end{eqnarray*}
where
\[
Z_{t}^{\alpha}=\alpha X_{t}+ ( 1-\alpha )
Y_{t}.
\]
Let us briefly show why we can indeed use the mean value theorem here.
We do it separately
for all three differences under the integrals in \eqref{e40'}. However,
we only explain it for the last difference. The other two can be
treated analogously. So, fix $i\in\N$. We want to prove that for $\gamma
$-a.e. starting point $z\in H$ we have $P\otimes dt$-a.e.
%
\begin{equation}
\label{e39'} Du^i(X_s)-Du^i(Y_s)
=\int_0^1D^2u^i
\bigl(\alpha X_s+(1-\alpha)Y_s\bigr)
(X_s-Y_s)\,d\alpha.
\end{equation}
We know by Corollary~\ref{c30new} and Remark~\ref{r33'} that there
exists $u_n\in\mathcal FC^2_b$, $n\in\N$, such that for $\lambda\ge
4\pi\|B\|^2_\infty$ and all $z\in S_V$
%
\begin{equation}
\label{e39''} \lim_{n\to\infty} \E_{P_z} \biggl[\int
_0^\infty e^{-\lambda
s}\bigl|Du^i-Du_n\bigr|^2
\bigl(X_s^V\bigr)\,ds \biggr]=0
\end{equation}
and
%
\begin{equation}
\label{e39'''} \lim_{n\to\infty} \int_H
\bigl\|D^2u^i-D^2u_n
\bigr\|^2_{HS} \,d\nu=0.
\end{equation}
Here, $P_z$ is from the Markov process
\[
\mathbf{M}:=\bigl(\Omega,\mathcal F, (\mathcal F_t)_{t\ge
0},
\bigl(X^V_t\bigr)_{t\ge0}, (P_z)_{z\in S_V}
\bigr)
\]
in Corollary~\ref{c30new} [and we changed notation and used
$(X^V_t)_{t\ge0}$ instead of $(X_t)_{t\ge0}$ in Corollary~\ref
{c30new} to avoid confusion with our fixed solution
$(X_t)_{t\in[0,T]}$ above].

Recalling that by Girsanov's theorem both $X$ and $Y$ have laws which
are equivalent to the law of $X^V:=X^V_t, t\in[0,T]$, it follows by
\eqref{e39''} that as $n\to\infty$
\[
\int_0^T\bigl|Du^i(X_s)-Du_n(X_s)\bigr|^2
\,ds\to0, \qquad \int_0^T\bigl|Du^i(Y_s)-Du_n(Y_s)\bigr|^2
\,ds\to0,
\]
in probability. If we can show that also
%
\begin{equation}
\label{e39''''} \int_0^T\int
_0^1\bigl\|D^2u^i-D^2u_n
\bigr\|_{HS} \bigl(\alpha X_s+(1-\alpha )Y_s
\bigr)|X_s-Y_s|\,d\alpha \,ds\to0
\end{equation}
in probability as $n\to\infty$, \eqref{e39'} follows, since it
trivially holds for $u_n$ replacing $u^i$.

But the expression in \eqref{e39''''} is bounded by
\[
\sup_{s\in[0,T]} |X_s-Y_s|\int
_0^T\int_0^1
\bigl\|D^2u^i-D^2u_n\bigr\|_{HS}
\bigl(\alpha X_s+(1-\alpha)Y_s\bigr)\,d\alpha \,ds
\]
and by the continuity of sample paths
\[
\sup_{s\in[0,T]} |X_s-Y_s|<\infty, \qquad P
\mbox{-a.s.}
\]
Furthermore, it follows from \eqref{e39'''} and the proof of Lemma~\ref
{l39} and Corollary~\ref{c40} below
that for $\nu$-a.e. $z\in S_V$
\[
\int_0^T\int_0^1
\bigl\|D^2u^i-D^2u_n\bigr\|_{HS}
\bigl(\alpha X_s+(1-\alpha )Y_s\bigr)\,d\alpha \,ds\to0
\]
as $n\to\infty$ $P$-a.s. Hence, \eqref{e39''''} follows.

By assumption \eqref{e2} in (H3) we know that
\[
\int_{H}\bigl\llVert D^{2}V ( x ) \bigr\rrVert
_{\mathcal{L} (
H ) }\nu ( dx ) <\infty
\]
and by Theorem~\ref{t33'} we know that
\[
\int_{H}\sum_{i=1}^{\infty}
\bigl( \lambda_{i}\bigl\llvert Du^{i} ( x ) \bigr\rrvert
^{2}+\bigl\llVert D^{2}u^{i} ( x ) \bigr\rrVert
_{HS}^{2} \bigr) \nu ( dx ) <\infty.
\]
Thus, we may apply Corollary~\ref{c40} below with
\[
f ( x ) =\bigl\llVert D^{2}V ( x ) \bigr\rrVert _{\mathcal{L} ( H ) }+\sum
_{i=1}^{\infty} \bigl( 2\lambda _{i}
\bigl\llvert Du^{i} ( x ) \bigr\rrvert ^{2}+\bigl\llVert
D^{2} u^{i} ( x ) \bigr\rrVert _{HS}^{2}
\bigr)
\]
and get that $\int_{0}^{1}\int_{0}^{t}f ( Z_{s}^{\alpha} )
\,d\alpha
\,ds<\infty$ with probability one, for every $t\geq0$ and $\nu$-a.e. $z\in
S^V$, that is,
\[
N_{t,z}<\infty
\]
with probability one, for every $t\geq0$, which completes the proof since
$A_{t}\leq4N_{t,z}$.

Now let us consider the case when (H3)$'$ holds. Clearly, we then handle
the second and the third term in the right-hand side of \eqref{e38'} as
above. For the first term, the treatment is different, but simpler.
Indeed, we have by (H3)$'$, Lemma~\ref{l36} and by the mean value theorem that
\begin{eqnarray*}&& \int_0^T
\frac{|\nabla V(X_s)-\nabla V(Y_s)|}{|\varphi(X_s)-\varphi
(Y_s)|} 1_{\varphi(X_s)\neq\varphi(Y_s)} \,ds
\\
&&\qquad\le2\int_0^T\int_0^1
\bigl\|V''_E\bigl(Z^\alpha_s
\bigr)\bigr\|_{L(H,E')} \,d\alpha \,ds
\\
&&\qquad\le2\int_0^T\bigl[\Psi\bigl(|X_s|_E\bigr)+
\Psi\bigl(|Y_s|_E\bigr)\bigr] \,ds.
\end{eqnarray*}
But again using Girsanov's theorem we know that the laws of $X$ and $Y$
are equivalent to that of $X^V$, hence the last expression is finite $P$-a.e.
\end{pf}

We may now prove Theorem~\ref{t4}. Let $z\in\Xi$. By Lemma~\ref{l36},
\begin{eqnarray*}&&d\bigl(\varphi^i(X_t)-
\varphi^i(Y_t)\bigr)\\
&&\qquad=-\bigl(\lambda _i
\bigl(X^i_t-Y^i_t
\bigr)+D_iV(X_t)-D_iV(Y_t)
\bigr)\,dt
\\
&&\qquad\quad{}+(\lambda+\lambda_i) \bigl(u^i(X_t)-u^i(Y_t)
\bigr)\,dt+ \bigl\langle Du^i(X_t)-Du^i(Y_t),dW_t
\bigr\rangle.
\end{eqnarray*}
Hence, by It\^{o}'s formula, we get
\begin{eqnarray*} && d\bigl(\varphi^i(X_t)-
\varphi^i(Y_t)\bigr)^2
\\
&&\qquad=-2\bigl(\varphi^i(X_t)-\varphi^i(Y_t)
\bigr) \bigl(\lambda _i\bigl(X^i_t-Y^i_t
\bigr)+D_iV(X_t)-D_iV(Y_t)
\bigr)\,dt
\\
&&\qquad\quad{}+2\bigl(\varphi^i(X_t)-\varphi^i(Y_t)
\bigr) (\lambda+\lambda _i) \bigl(u^i(X_t)-u^i(Y_t)
\bigr)\,dt
\\
&&\qquad\quad{}+2\bigl(\varphi^i(X_t)-\varphi^i(Y_t)
\bigr) \bigl\langle Du^i(X_t)-Du^i(Y_t),dW_t
\bigr\rangle
\\
&&\qquad\quad{}+ \bigl|Du^i(X_t)-Du^i(Y_t)\bigr|^2\,dt.
\end{eqnarray*}
By definition of $\varphi$ in Lemma~\ref{l36},
in the lines above there are the terms $-2 ( u^{i} (
X_{t} ) -u^{i} ( Y_{t} )  ) \lambda_{i} (
X_{t}^{i}-Y_{t}^{i} ) $ and $2 ( X_{t}^{i}-Y_{t}^{i} )
\lambda_{i} ( u^{i} ( X_{t} ) -u^{i} ( Y_{t} )
 ) $ which cancel each other. Moreover, the term $-2 ( X_{t}%
^{i}-Y_{t}^{i} ) \lambda_{i} ( X_{t}^{i}-Y_{t}^{i} ) $ is
negative. Thus, we deduce
\begin{eqnarray*}
d \bigl( \varphi^{i} ( X_{t} ) -\varphi^{i} (
Y_{t} ) \bigr) ^{2} & \leq&-2 \bigl( \varphi^{i}
( X_{t} ) -\varphi ^{i} ( Y_{t} ) \bigr) \bigl(
D_{i}V ( X_{t} ) -D_{i}V ( Y_{t} )
\bigr) \,dt
\\
&&{} +2\lambda \bigl( \varphi^{i} ( X_{t} ) -
\varphi^{i} ( Y_{t} ) \bigr) \bigl( u^{i} (
X_{t} ) -u^{i} ( Y_{t} ) \bigr) \,dt
\\
&&{} +2\lambda_{i} \bigl( u^{i} ( X_{t} )
-u^{i} ( Y_{t} ) \bigr) ^{2}\,dt
\\
&&{} +2 \bigl( \varphi^{i} ( X_{t} ) -\varphi^{i}
( Y_{t} ) \bigr) \bigl\langle Du^{i} ( X_{t} )
-Du^{i} ( Y_{t} ) ,dW_{t} \bigr\rangle
\\
& &{}+\bigl\llvert Du^{i} ( X_{t} ) -Du^{i} (
Y_{t} ) \bigr\rrvert ^{2}\,dt.
\end{eqnarray*}
Let $A_{t}=A_{t,z}$ be the process introduced in Lemma~\ref{l38}. We have
\begin{eqnarray*} &&d\bigl(e^{-A_t}\bigl(\varphi^i(X_t)-
\varphi^i(Y_t)\bigr)^2\bigr)
\\
&&\qquad\le-2e^{-A_t}\bigl(\varphi^i(X_t)-
\varphi^i(Y_t)\bigr) \bigl(D_iV(X_t)-D_iV(Y_t)
\bigr) \,dt
\\
&&\qquad\quad{}+2\lambda e^{-A_t}\bigl(\varphi^i(X_t)-\varphi
^i(Y_t)\bigr) \bigl(u^i(X_t)-u^i(Y_t)
\bigr) \,dt
\\
&&\qquad\quad{}+2e^{-A_t}\bigl(\varphi^i(X_t)-
\varphi^i(Y_t)\bigr)\bigl\langle Du^i(X_t)-Du^i(Y_t),dW_t
\bigr\rangle
\\
&&\qquad\quad{}+2\lambda_ie^{-A_t}\bigl(u^i(X_t)-u^i(Y_t)
\bigr)^2 \,dt
\\
&&\qquad\quad{}+e^{-A_t}\bigl|Du^i(X_t)-Du^i(Y_t)\bigr|^2
\,dt-e^{-A_t}\bigl(\varphi ^i(X_t)-
\varphi^i(Y_t)\bigr)^2 \,dA_t
\end{eqnarray*}
and thus, for every $N>0$, summing the previous inequality for
$i=1,\ldots,N$, we get
\begin{eqnarray*}
&&d\bigl(e^{-A_t}\bigl|P_N\bigl(
\varphi(X_t)-\varphi(Y_t)\bigr)\bigr|^2\bigr)
\\
&&\qquad\le-2e^{-A_t}\langle P_N\bigl(\varphi(X_t)-
\varphi(Y_t)\bigr),P_N\bigl(\nabla V(X_t)-
\nabla V(Y_t)\bigr)\bigr\rangle \,dt
\\
&&\qquad\quad{}+2\lambda e^{-A_t}\bigl\langle P_N\bigl(
\varphi(X_t)-\varphi(Y_t)\bigr), U(X_t)-U(Y_t)
\bigr\rangle \,dt
\\
&&\qquad\quad{}+2e^{-A_t}\sum_{i=1}^N\bigl(
\varphi^i(X_t)-\varphi ^i(Y_t)
\bigr)\bigl\langle Du^i(X_t)-Du^i(Y_t),dW_t
\bigr\rangle
\\
&&\qquad \quad{}+2e^{-A_t}\sum_{i=1}^N
\lambda_i\bigl(u^i(X_t)-u^i(Y_t)
\bigr)^2 \,dt
\\
&&\qquad\quad{}+e^{-A_t}\sum_{i=1}^N\bigl|Du^i(X_t)-Du^i(Y_t)\bigr|^2
\,dt\\
&&\qquad\quad{}-e^{-A_t}\bigl|P_N\bigl(\varphi(X_t)-
\varphi(Y_t)\bigr)\bigr|^2 \,dA_t.
\end{eqnarray*}
Substituting $dA_{t}$, taking expectation and using simple inequalities
we get
\begin{eqnarray*}
&&\E \bigl[ e^{-A_{t}}\bigl\llvert
P_{N} \bigl( \varphi ( X_{t} ) -\varphi ( Y_{t}
) \bigr) \bigr\rrvert ^{2} \bigr]
\\
&&\qquad\leq 2\lambda\int_{0}^{t}\E \bigl[
e^{-A_{s}}\bigl\llvert \varphi ( X_{s} ) -\varphi (
Y_{s} ) \bigr\rrvert \bigl\llvert U ( X_{s} ) -U (
Y_{s} ) \bigr\rrvert \bigr] \,ds \\
&&\qquad\quad{} +2\int_{0}^{t}\E \bigl[ e^{-A_{s}}\bigl
\llvert \varphi ( X_{s} ) -\varphi ( Y_{s} ) \bigr\rrvert
\bigl\llvert P_{N} \bigl( \nabla V ( X_{s} ) -\nabla V (
Y_{s} ) \bigr) \bigr\rrvert \bigr] \,ds
\\
&&\qquad\quad{} -2\int_{0}^{t}\E \biggl[ e^{-A_{s}}
\bigl\llvert P_{N} \bigl( \varphi ( X_{s} ) -\varphi (
Y_{s} ) \bigr) \bigr\rrvert ^{2} \\
&&\hspace*{80pt}{}\times\frac{ 2\llvert  \nabla
V ( X_{s} ) -\nabla V ( Y_{s} ) \rrvert
}{\llvert \varphi ( X_{s} ) -\varphi ( Y_{s} )
\rrvert }1_{\varphi ( X_{s} ) \neq\varphi ( Y_{s} )
}
\biggr] \,ds
\\
&&\qquad\quad{}+\int_{0}^{t}\E \bigl[ e^{-A_{s}}g_{s}
\bigr] \,ds\\
&&\qquad\quad{}-\int_{0}^{t}E \biggl[
e^{-A_{s}}g_{s}\frac{\llvert  P_{N} ( \varphi ( X_{s} )
-\varphi ( Y_{s} )  ) \rrvert ^{2}}{\llvert
\varphi ( X_{s} ) -\varphi ( Y_{s} ) \rrvert ^{2}
}1_{\varphi ( X_{s} ) \neq\varphi ( Y_{s} ) } \biggr] \,ds,
\end{eqnarray*}
where for shortness of notation we have written
\[
g_{s}:=2\sum_{i=1}^{\infty}
\lambda_{i} \bigl( u^{i} ( X_{s} )
-u^{i} ( Y_{s} ) \bigr) ^{2}+\sum
_{i=1}^{\infty}\bigl\llvert Du^{i} (
X_{s} ) -Du^{i} ( Y_{s} ) \bigr\rrvert
^{2}.
\]
By monotone convergence, we may take the limit as $N\rightarrow\infty$
and deduce
\begin{eqnarray*}
&&\E \bigl[ e^{-A_{t}}\bigl\llvert \varphi ( X_{t} ) -\varphi (
Y_{t} ) \bigr\rrvert ^{2} \bigr] \\
&&\qquad\leq2\lambda\int
_{0}^{t}\E \bigl[ e^{-A_{s}}\bigl\llvert
\varphi ( X_{s} ) -\varphi ( Y_{s} ) \bigr\rrvert \bigl
\llvert U ( X_{s} ) -U ( Y_{s} ) \bigr\rrvert \bigr] \,ds
\\
&&\quad\qquad{} +2\int_{0}^{t}\E \bigl[ e^{-A_{s}}\bigl
\llvert \varphi ( X_{s} ) -\varphi ( Y_{s} ) \bigr\rrvert
\bigl\llvert \nabla V ( X_{s} ) -\nabla V ( Y_{s} ) \bigr
\rrvert \bigr] \,ds
\\
&&\qquad\quad{} -2\int_{0}^{t}\E \biggl[ e^{-A_{s}}
\bigl\llvert \varphi ( X_{s} ) -\varphi ( Y_{s} ) \bigr
\rrvert ^{2}\frac{ 2\llvert  \nabla
V ( X_{s} )
-\nabla V ( Y_{s} ) \rrvert }{\llvert
\varphi ( X_{s} ) -\varphi ( Y_{s} ) \rrvert
}1_{\varphi ( X_{s} ) \neq\varphi ( Y_{s} ) } \biggr] \,ds
\\
&&\qquad\quad{}+\int_{0}^{t}\E \bigl[ e^{-A_{s}}g_{s}
\bigr] \,ds-\int_{0}^{t}\E \biggl[
e^{-A_{s}}g_{s}\frac{\llvert \varphi ( X_{s} ) -\varphi (
Y_{s} ) \rrvert ^{2}}{\llvert \varphi ( X_{s} )
-\varphi ( Y_{s} ) \rrvert ^{2}}1_{\varphi (
X_{s} )
\neq\varphi ( Y_{s} ) } \biggr] \,ds.
\end{eqnarray*}
Notice that by Lemma~\ref{l37}, $X_s=Y_s$ if and only if $\varphi
(X_s)=\varphi(Y_s)$. Hence, we may drop the indicator function
$1_{\varphi ( X_{s} )
\neq\varphi ( Y_{s} ) }$ in all integrals in the above inequality.

Therefore, certain
terms cancel in the previous inequality and we get
\[
\E \bigl[ e^{-A_{t}}\bigl\llvert \varphi ( X_{t} ) -\varphi (
Y_{t} ) \bigr\rrvert ^{2} \bigr] \leq2\lambda\int
_{0}^{t}\E \bigl[ e^{-A_{s}}\bigl\llvert
\varphi ( X_{s} ) -\varphi ( Y_{s} ) \bigr\rrvert \bigl
\llvert U ( X_{s} ) -U ( Y_{s} ) \bigr\rrvert \bigr] \,ds.
\]
Using Lemmas \ref{lemma uniform est} and \ref{l37}, we
get
\[
\E \bigl[ e^{-A_{t}}\bigl\llvert \varphi ( X_{t} ) -\varphi (
Y_{t} ) \bigr\rrvert ^{2} \bigr] \leq2\lambda C\int
_{0}^{t}\E \bigl[ e^{-A_{s}}\bigl\llvert
\varphi ( X_{s} ) -\varphi ( Y_{s} ) \bigr\rrvert
^{2} \bigr] \,ds,
\]
whence $\E [ e^{-A_{t}}\llvert \varphi ( X_{t} )
-\varphi ( Y_{t} ) \rrvert ^{2} ] =0$ by Gronwall's lemma,
and thus $\varphi ( X_{t} ) =\varphi ( Y_{t} ) $ with
probability one (since $A_{t}<\infty$ a.s.), for all $t\geq0$; the
same is
true for the identity $X_{t}=Y_{t}$ since $\varphi$ is invertible and finally
$X$ and $Y$ are also indistinguishable since they are continuous processes.

To complete the proof, we have to prove Corollary~\ref{c40} below,
which was
used in the proof of Lemma~\ref{l38}.

\section{Main lemmata}\label{sec6}

Let $S_V$ as in Remark~\ref{r34'} and $H_V$ as in \eqref{e9'} and set
%
\begin{equation}
\label{e40} \Xi_V:=S_V\cap H_V.
\end{equation}
%

\begin{lemma}
\label{l39}Let $f\dvtx H\rightarrow{}[0,\infty)$ be a Borel
measurable function such that
%
\begin{equation}
\int_{H}f ( x ) \gamma ( dx ) <\infty .\label{assump on f}
\end{equation}
Then there is a Borel set $\Xi\subset S_V\cap H_V$ with $\gamma (
\Xi ) =1$
having the following property. Given any $z\in\Xi$ and any two
solutions $X,Y$
with initial condition $z$ (as in the statement of Theorem~\ref{t4}) for all $T>0$ we have
\[
P \biggl( \int_{0}^{1}\int_{0}^{T}f
\bigl( Z_{s}^{\alpha} \bigr) \,ds\,d\alpha<\infty \biggr) =1,
\]
where $Z_{t}^{\alpha}=\alpha X_{t}+ ( 1-\alpha ) Y_{t}$.
\end{lemma}

\begin{pf}
\textit{Step} 1 (Estimates on OU process). A number $T>0$ is fixed throughout
the proof. From the assumption on $f$, it follows that there is a Borel set
$\Xi_{f}\subset H$, with $\Xi_{f}^{c}$ of $\gamma$-measure zero, such that
\[
\E \biggl[ \int_{0}^{T}f \bigl(
Z_{s}^{\mathrm{OU},z} \bigr) \,ds \biggr] =\int_{0}
^{T} \biggl( \int_{H}f ( x ) p_{s,z} (
dx ) \biggr) \,ds<\infty
\]
for all $z\in\Xi_{f}$, where $p_{s,z} ( dx ) $ is the law at time
$s$ of the Ornstein--Uhlenbeck process $Z_{s}^{\mathrm{OU},z}$, that is, the
solution of the equation
%
\begin{equation}
dZ_{t}=AZ_{t}\,dt+dW_{t},\qquad Z_{0}=z.
\label{OU SDE}
\end{equation}
Indeed, we have
\begin{eqnarray*}
&& \int_{H} \biggl( \int_{0}^{T}
\biggl( \int_{H}f ( x ) p_{s,z} ( dx ) \biggr) \,ds
\biggr) \gamma ( dz )
\\
&&\qquad =\int_{0}^{T} \biggl( \int
_{H}\int_{H}f ( x ) p_{s,z}
( dx ) \gamma ( dz ) \biggr) \,ds
\\
&&\qquad =\int_{0}^{T} \biggl( \int
_{H}f ( z ) \gamma ( dz ) \biggr) \,ds=T\int
_{H}f ( z ) \gamma ( dz ).
\end{eqnarray*}
This implies $\int_{0}^{T} ( \int_{H}f ( x ) p_{s,z} (
dx )  ) \,ds<\infty$ for $\gamma$-a.e. $z$.

\textit{Step} 2 (Girsanov transform). Let $\Xi_{V}$ as in \eqref{e40}
and $\Xi_{f}$ be given as in step 1. Let
$\Xi=\Xi_{V}\cap\Xi_{f}$, of full $\gamma$-measure. In the sequel, $z\in
\Xi$ will
be given, thus we avoid to index all quantities by $z$.

From Theorem~\ref{t2.10}, we have
\[
\int_{0}^{T}\bigl\llvert \nabla
V(X_{s})\bigr\rrvert ^{2} \,ds+\int_{0}
^{T}\bigl\llvert \nabla V(Y_{s})\bigr\rrvert
^{2} \,ds<\infty
\]
for all $T>0$, with probability one.

Let us introduce the sequence $\{\tau^{n}\}$ of stopping times defined
as
\begin{eqnarray*}
\tau^{n}&=&\tau_{B}^{n}\wedge
\tau_{V,1}^{n}\wedge\tau_{V,2}^{n},%
\\
\tau_{B}^{n}&:=&\inf \biggl\{ t\geq0\dvtx \biggl\llvert
\int_{0}^{t}B(X_{s}%
)
\,dW_{s}\biggr\rrvert +\biggl\llvert \int_{0}^{t}B(Y_{s})
\,dW_{s}\biggr\rrvert \geq n \biggr\} \wedge T,
\\
\tau_{V,1}^{n} & :=&\inf \biggl\{ t\geq0\dvtx \biggl\llvert
\int_{0}^{t}\bigl\langle \nabla
V(X_{s}),dW_{s}\bigr\rangle\biggr\rrvert +\biggl\llvert
\int_{0}^{t}\bigl\langle \nabla
V(Y_{s} ),dW_{s}\bigr\rangle\biggr\rrvert \geq n \biggr\}
\wedge T,
\\
\tau_{V,2}^{n} & :=&\inf \biggl\{ t\geq0\dvtx \int
_{0}^{t}\bigl\llvert \nabla V(X_{s})
\bigr\rrvert ^{2} \,ds+\int_{0}^{t}\bigl
\llvert \nabla V(Y_{s})\bigr\rrvert ^{2} \,ds\geq n \biggr\}
\wedge T
\end{eqnarray*}
for $n\geq1$ (an infimum is equal to $+\infty$ if the corresponding set is
empty). All stochastic and Lebesgue integrals are well defined and continuous
in $t$, hence we have $\tau^{n}=T$ eventually, with probability one. In order
to prove the lemma, it is sufficient to prove that $E [ \int_{0}^{1}
\int_{0}^{T\wedge\tau^{n}}f ( Z_{s}^{\alpha} ) \,ds\,d\alpha ]
<\infty$ for each $n$.

Let us also introduce the stochastic processes
\begin{eqnarray*}
b_{s}^{\alpha} & :=&\alpha B(X_{s})+(1-
\alpha)B(Y_{s}),
\\
v_{s}^{\alpha} & :=&\alpha\nabla V(X_{s})+(1-
\alpha)\nabla V(Y_{s})
\end{eqnarray*}
and the stochastic exponentials
\[
\rho_{t}^{\alpha}:=\exp \biggl( -\int_{0}^{t}
\bigl\langle b_{s}^{\alpha}%
-v_{s}^{\alpha},dW_{s}
\bigr\rangle-\frac{1}{2}\int_{0}^{t}\bigl|b_{s}^{\alpha
}-v_{s}^{\alpha}\bigr|^{2}
\,ds \biggr).
\]
Denote
\[
\rho_{t}^{\alpha,n}:=\rho_{t\wedge\tau^{n}}^{\alpha}=\exp
\biggl( -\int_{0}%
^{t} \bigl
\langle1_{s\leq\tau^{n}} \bigl( b_{s}^{\alpha}-v_{s}^{\alpha
}
\bigr) ,dW_{s} \bigr\rangle-\frac{1}{2}\int
_{0}^{t}1_{s\leq\tau^{n}}%
\bigl|b_{s}^{\alpha}-v_{s}^{\alpha}\bigr|^{2}
\,ds \biggr).
\]
By Novikov's criterium, this is a martingale (indeed $\int_{0}^{T}1_{s\leq
\tau^{n}}|b_{s}^{\alpha}-v_{s}^{\alpha}|^{2} \,ds$ is a bounded r.v. We
may thus introduce the following new measures (and the
corresponding expectations)
\[
Q^{\alpha,n}(A):=\E \bigl[ \rho_{T}^{\alpha,n}1_{A}
\bigr].
\]
Girsanov's theorem implies that
\begin{eqnarray*}
\widetilde{W}_{t}^{n,\alpha} & :=&W_{t}+\int
_{0}^{t}1_{s\leq\tau
^{n}} \bigl(
b_{s}^{\alpha}-v_{s}^{\alpha} \bigr) \,ds
\\
& =&W_{t}+\int_{0}^{t\wedge\tau^{n}} \bigl(
b_{s}^{\alpha}-v_{s}^{\alpha
} \bigr) \,ds
\end{eqnarray*}
is a new cylindrical Brownian motion.

\textit{Step} 3 (Auxiliary process and conclusion). Recall also that
$Z_{t}^{\alpha}$ (with the new notation) satisfies
\[
dZ_{t}^{\alpha}=AZ_{t}^{\alpha}\,dt+ \bigl(
b_{s}^{\alpha}-v_{s}^{\alpha
} \bigr)
\,dt+dW_{t}.
\]
Let us introduce the auxiliary process $Z_{t}^{\alpha,n}$ which solves,
in the
sense of Definition~\ref{Def sol}, the equation
\[
Z_{t}^{\alpha,n}=z+\int_{0}^{t}AZ_{s}^{\alpha,n}\,ds+
\int_{0}^{t}1_{s\leq
\tau^{n}} \bigl(
b_{s}^{\alpha}-v_{s}^{\alpha} \bigr)
\,ds+W_{t}.
\]
It exists, by the explicit formula
\[
Z_{t}^{\alpha,n}=e^{tA}z+\int_{0}^{t}e^{ ( t-s ) A}1_{s\leq\tau
^{n}}
\bigl( b_{s}^{\alpha}-v_{s}^{\alpha} \bigr)
\,ds+\int_{0}^{t}e^{ (
t-s ) A}\,dW_{s},
\]
where $e^{tA}$ is the analytic semigroup in $H$ generated by $A$
(taking inner
product with the elements $e_{k}$ of the basis, it is not difficult to check
that this mild formula gives a solution in the weak sense of Definition~\ref{Def sol}). This process satisfies also
\[
Z_{t}^{\alpha,n}=z+\int_{0}^{t}AZ_{s}^{\alpha,n}\,ds+
\widetilde{W}_{t}%
^{n,\alpha}%
\]
by the definition of $\widetilde{W}_{t}^{n,\alpha}$, hence its law under
$Q^{\alpha,n}$ is the same as the Gaussian law of $Z_{t}^{\mathrm{OU}}$ under $P$.
Moreover,
\[
Z_{t}^{\alpha,n}=Z_{t}^{\alpha}\qquad\mbox{for }t\in
\bigl[ 0,\tau^{n} \bigr]
\]
(indeed, by the weak formulation, the process $Y_{t}=Z_{t}^{\alpha,n}%
-Z_{t}^{\alpha}$ verifies, pathwise, on $ [ 0,\tau^{n} ] $, the
equation $Y_{t}^{\prime}=AY_{t}$, $Y_{0}=0$, in the weak sense of Definition~\ref{Def sol} and thus, taking inner product with the elements $e_{k}$
of the
basis, one proves $Y=0$).

Therefore,
\begin{eqnarray*}
\E^{Q^{\alpha,n}} \biggl[ \int_{0}^{T\wedge\tau^{n}}f \bigl(
Z_{s}^{\alpha
} \bigr) \,ds \biggr]& =&\E^{Q^{\alpha,n}} \biggl[
\int_{0}^{T\wedge\tau^{n}%
}f \bigl( Z_{s}^{\alpha,n}
\bigr) \,ds \biggr]
\\
&\leq&\E^{Q^{\alpha
,n}} \biggl[ \int_{0}^{T}f
\bigl( Z_{s}^{\alpha,n} \bigr) \,ds \biggr]
\\
&=&\E \biggl[ \int_{0}^{T}f \bigl(
Z_{s}^{\mathrm{OU}} \bigr) \,ds \biggr] =:C^{\prime
}<\infty.
\end{eqnarray*}
But
\begin{eqnarray*}
\E^{Q^{\alpha,n}} \biggl[ \int_{0}^{T\wedge\tau^{n}}f \bigl(
Z_{s}^{\alpha
} \bigr) \,ds \biggr] & =&\E \biggl[
\rho_{T}^{\alpha,n}\int_{0}^{T\wedge
\tau^{n}}f
\bigl( Z_{s}^{\alpha} \bigr) \,ds \biggr]
\\
& \geq& C_{n}\E \biggl[ \int_{0}^{T\wedge\tau^{n}}f
\bigl( Z_{s}^{\alpha
} \bigr) \,ds \biggr],
\end{eqnarray*}
where $C_{n}>0$ is a constant such that $\rho_{T}^{\alpha,n}\geq
C_{n}$: it
exists because
\[
\bigl( \rho_{T}^{\alpha,n} \bigr) ^{-1}:=\exp \biggl(
\int_{0}^{T\wedge
\tau^{n}} \bigl\langle b_{s}^{\alpha}-v_{s}^{\alpha},dW_{s}
\bigr\rangle +\frac{1}{2}\int_{0}^{T\wedge\tau^{n}}\bigl|b_{s}^{\alpha}-v_{s}^{\alpha}%
\bigr|^{2}
\,ds \biggr)
\]
and $\tau^{n}$ includes the stopping of all these integrals. Therefore,
\[
\E \biggl[ \int_{0}^{T\wedge\tau^{n}}f \bigl(
Z_{s}^{\alpha} \bigr) \,ds \biggr] \leq\frac{C^{\prime}}{C_{n}}%
\]
and thus also
\[
\E \biggl[ \int_{0}^{1}\int
_{0}^{T\wedge\tau^{n}}f \bigl( Z_{s}^{\alpha
}
\bigr) \,ds\,d\alpha \biggr] \leq\frac{C^{\prime}}{C_{n}}.
\]
The proof is complete.
\end{pf}

The next corollary extends the previous result to the case when\break \mbox{$\int%
_{H}f ( x ) \nu ( dx ) <\infty$}. Clearly,
\[
\int_{H}f ( x ) \nu ( dx ) \leq\frac{1}{Z}
\int%
_{H}f ( x ) \gamma ( dx )
\]
but not conversely, without additional assumptions on $V$. Hence, Corollary~\ref{c40} implies Lemma~\ref{l39}, but
not conversely, in an obvious way. However, we may easily deduce Corollary~\ref{c40} from Lemma~\ref{l39} by assumptions (H1)--(H3).

\begin{corollary}
\label{c40}Let $f\dvtx H\rightarrow{}[0,\infty)$ be a
Borel measurable function such that
%
\begin{equation}
\int_{H}f ( x ) \nu ( dx ) <\infty .\label{assump on f non gaussian}
\end{equation}
Then there is a Borel set $\Xi\subset S_V\cap H_V$ with $\nu ( \Xi
 ) =1$
[equivalently $\gamma ( \Xi ) =1$] having the property stated in
Lemma~\ref{l39}.
\end{corollary}

\begin{pf}
Since $\int_{H}f ( x ) e^{-V ( x ) }\gamma (
dx ) <\infty$, we may apply Lemma~\ref{l39} to the
function $f ( x ) e^{-V ( x ) }$ instead of $f (
x ) $ and get, as a result, that there is a Borel set $\Xi\subset
S_V\cap H_V$
with $\gamma ( \Xi ) =1$ having the following property: given any
$z\in\Xi$ and any two solutions $X,Y$ as in the statement of Theorem~\ref{t4}, for all $T>0$ we have
%
\begin{equation}
\label{e21} P \biggl( \int_{0}^{1}\int
_{0}^{T}f \bigl( Z_{s}^{\alpha}
\bigr) e^{-V (
Z_{s}^{\alpha} ) }\,ds\,d\alpha<\infty \biggr) =1,
\end{equation}
where $Z_{t}^{\alpha}=\alpha X_{t}+ ( 1-\alpha ) Y_{t}$. Take
$z\in\Xi$. Since $V(Z^\alpha_s)\le V(X_t)+V(Y_t)$ (recall that $V\ge0$
by Remark~\ref{r0}) and by Theorem~\ref{t2.10},
\[
P \Biggl(\bigcup_{n=1}^\infty \bigl\{
\sigma^{X,Y}_{H\setminus K^V_n} >T \bigr\} \Biggr) =1,
\]
where
\[
\sigma^{X,Y}_{H\setminus K^V_n}:=\min \bigl(\sigma^{X}_{H\setminus
K^V_n},
\sigma^{Y}_{H\setminus K^V_n} \bigr)
\]
and $\sigma^{X}_{H\setminus K^V_n},\sigma^{Y}_{H\setminus K^V_n}$ are
the first hitting times of $H\setminus K^V_n$ of $X,Y$, respectively, we
have by \eqref{e21}
\[
\int_0^1\int_0^T
f\bigl(Z^\alpha_s\bigr) e^{-V(X_s)} e^{-V(Y_s)} \,ds
\,d\alpha <\infty\qquad \mbox{on } \bigcup_{n=1}^\infty
\bigl\{\sigma ^{X,Y}_{H\setminus K^V_n} >T \bigr\}, P\mbox{-a.s.}
\]
But for $\omega\in\{\sigma^{X,Y}_{H\setminus K^V_n}>T\}$ and $M_n:=\sup
\{V(z)\dvtx  z\in K_n^V\}$
\[
 \int_0^1\int
_0^T f\bigl(Z^\alpha_s(
\omega)\bigr) \,ds \,d\alpha \le e^{2M_n} \int_0^1
\int_0^T f\bigl(Z^\alpha_s(
\omega)\bigr) e^{-V(X_s)} e^{-V(Y_s)} \,ds \,d\alpha<\infty.
\]
Hence,
\[
P \biggl( \int_0^1\int_0^T
f\bigl(Z^\alpha_s(\omega)\bigr) \,ds \,d\alpha<\infty \biggr)
=1.
\]
\upqed\end{pf}
%

\section{Applications}\label{sec7}

\subsection{Reaction--diffusion equations}\label{sec7.1}

Let $H:=L^2((0,1),d\xi)$, with $d\xi=  \mathrm{Lebesgue}$ measure and $A=-\Delta$
with domain $H^2(0,1)\cap H^1_0(0,1)$, that is, $A$ is the Dirichlet
Laplacian on $(0,1)$. Then clearly (H1) holds.\vadjust{\goodbreak}

Let $m\in[1,\infty)$ and
%
\begin{equation}
\label{e51'} V(x):=\cases{ %
\displaystyle\int
_0^1\bigl|x(\xi)\bigr|^{m+1} \,d\xi,&\quad
$\mbox{if } x\in L^{m+1}\bigl((0,1),d\xi\bigr),$
\vspace*{2pt}\cr
+\infty,&\quad $\mbox{else}.$}
\end{equation}
$V$ obviously satisfies (H2). Now we are going to verify (H3)$'$ for this
convex functional. (Of course, then according to Remark~\ref{r0} we
subsequently replace this $V$ by $V+\tfrac{\omega}2 |\cdot
|^2_H$.)

For the separable Banach space $E$ in (H3)$'$, we take
%
\begin{equation}
\label{e7.1} E:=L^{2m}\bigl((0,1),d\xi\bigr)=:L^{2m}.
\end{equation}
Then by elementary calculations for $x\in E$
%
\begin{eqnarray}
\label{e7.2} V'_E(x)&=&(m+1)|x|^{m-1} x\in H
\subset L^{{2m}/{(2m-1)}}=E',
\\
\label{e7.3} V''_E(x)
(h_1,h_2)&=&m(m+1)\int_0^1\bigl|x(
\xi)\bigr|^{m-1} h_1(\xi) h_2(\xi) \,d\xi,
\end{eqnarray}
for $h_1,h_2\in E$. Obviously, the right-hand side of \eqref{e7.3} is
also defined for $h_1,h_2\in H$ and by
H\"older's inequality, continuous in $(h_1,h_2)\in E\times H$ with
respect to the product topology. Hence, for all $x\in E$
\[
V''_E(x)\subset L\bigl(H,E'
\bigr)
\]
and furthermore (again by H\"older's inequality)
\[
\bigl\|V''_E(x)\bigr\|_{L(H,E')}
\le|x|^{m-1}_E.
\]
Equation~\eqref{e7.2} implies that $E\subset D_V$. But our Gaussian measure
$\gamma=N_{-({1}/2) A^{-1}}$ is known to have full mass even on
$C([0,1];\R)$ because it is the law of the Brownian Bridge,
hence $\gamma(E)=1$ and so, $\gamma(D_V)=1$. Furthermore, then
obviously by Fernique's theorem the first inequality in \eqref{e2} is satisfied.

It remains to verify \eqref{e2'}, that is, for $\gamma$-a.e. initial
condition $z\in H$
%
\begin{equation}
\label{e7.4} \E\int_0^T\bigl|X^V(s)\bigr|^{m-1}_E
\,ds<\infty,
\end{equation}
where $X^V(t), t\in[0,T]$, solves SDE \eqref{SDE} with
$B=0$. But the existence of such a process for $\gamma$-a.e. $z\in H$
follows from Theorem~\ref{t2.5} in Section~\ref{sec2} above. That this process
satisfies \eqref{e7.4} follows from results in \cite{BDR10}. Indeed, it
follows by
\cite{BDR10}, Theorem~3.6 and Proposition~4.1, and Fatou's lemma that even
\[
\E\int_0^T\bigl|X^V(s)\bigr|^{2m}_E
\,ds<\infty
\]
for ($\gamma$-a.e.) $z\in E$.

Hence, (H3)$'$ is verified and our main result, Theorem~\ref{t4}, applies
to this case.

\subsection{Weakly differentiable drifts}\label{sec7.2}

The main motivation to also consider condition (H3), that is, to assume
that the ($\gamma$-weak) second derivative $D^2V$ of $V$ exists and is
in $L^1(H,\gamma;L(H))$, was to make a connection between our results
and those in finite dimensions by \cite{CJ13}. As mentioned in the
\hyperref[sec1]{Introduction}, our results generalize some of the results of \cite{CJ13}
in the special case when $H=\R^d$. In addition, since we work with
respect to a Gaussian measure (and not Lebesgue measure on $\R^d$) our
integrability conditions are generically weaker than those in \cite
{CJ13}. As far as the infinite dimensional case is concerned, one might
ask what are examples of such functions $V$ satisfying condition (H3).
There are plenty of them and let us briefly describe a whole class of
such functions.

Let $\varphi\dvtx H\to[0,\infty]$ be convex, lower semicontinuous, $\varphi
\in L^{2+\delta}(H,\gamma)$ for some $\delta>0$, and G\^ateaux
differentiable, $\gamma$-a.e., that is, $\gamma(D_\varphi)=1$. Define
%
\begin{equation}
\label{e8.1} V(x):=R\bigl(\lambda, \mathcal L^{\mathrm{OU}}\bigr)
\varphi(x),\qquad x\in H,
\end{equation}
with $R(\lambda, \mathcal L^{\mathrm{OU}})$ defined as in \eqref{e24'}, that is,
it is the resolvent of the Ornstein--Uhlenbeck operator $\mathcal
L^{\mathrm{OU}}$. Then it is elementary to check from the definition that $V\dvtx H\to
[0,\infty]$ is also convex and lower semicontinuous.

Furthermore, $V$ is in the $L^{2}(H,\gamma)$-domain of
$\mathcal L^{\mathrm{OU}}$. Hence, by the maximal regularity result of \cite
{DL14} (already recalled in Section~\ref{sec3.3} above) applied to the case when
$U\equiv0$, we conclude that
$V\in W^{2,2}(H,\gamma)$, in particular we have
\[
\int_H\bigl\|D^2V\bigr\|^2_{HS}
\,d\gamma<\infty,
\]
which is stronger than the second part of condition \eqref{e2} in (H3).

Of course, one needs additional, but obviously quite mild bounds on
$\nabla\varphi$, to ensure that $\gamma(D_V)=1$ and $\nabla V\in
L^{2}(H,\gamma)$. But then the class of $V$ defined in \eqref{e8.1}
satisfy (H3). To be concrete in choosing $\varphi$ above, consider the
situation of Section~\ref{sec7.1}. Then if we take $\varphi:=V$ as defined in
\eqref{e51'}, the new $V$ given by
\eqref{e8.1} satisfy (H3).

\section*{Acknowledgments}
F. Flandoli, M. R\"{o}ckner and A. Yu. Veretennikov
would like to thank his hosts
from the Scuola Normale Superiore and University of Pisa for a great
stay in Pisa in March 2014, where a large part of this work was done.

For the fourth author, the article
was prepared within the framework of a subsidy granted to the HSE by the
Government of the Russian Federation for the implementation of the Global
Competitiveness Program.


\printaddresses
\end{document}